\documentclass[11pt, letterpaper]{amsart} %{article}

\usepackage{geometry}
\usepackage{color}
\usepackage{verbatim}
\usepackage{amsmath,amssymb,amsthm}
\usepackage[format=hang]{caption}
\usepackage{color}

\newcommand*{\R}{{\mathbb R}}

\newcommand*{\N}{{\mathbb N}}

\newcommand*{\Q}{\mathbb{Q}}

\newcommand*{\eps}{\varepsilon}

\newcommand*{\pip}{\varphi}

\newcommand*{\Mod}{\text{Mod}_1}
\newcommand*{\AM}{\text{AM}}

\providecommand*{\vint}[1]{\mathchoice
          {\mathop{\vrule width 5pt height 3 pt depth -2.5pt
                  \kern -9pt \kern 1pt\intop}\nolimits_{\kern -5pt{#1}}}
          {\mathop{\vrule width 5pt height 3 pt depth -2.6pt
                  \kern -6pt \intop}\nolimits_{\kern -3pt{#1}}}
          {\mathop{\vrule width 5pt height 3 pt depth -2.6pt
                  \kern -6pt \intop}\nolimits_{\kern -3pt{#1}}}
          {\mathop{\vrule width 5pt height 3 pt depth -2.6pt
                  \kern -6pt \intop}\nolimits_{\kern -3pt{#1}}}}

\DeclareMathOperator{\Lip}{Lip}

\DeclareMathOperator{\dist}{dist}
\DeclareMathOperator{\diam}{diam}
\DeclareMathOperator{\rad}{rad}

\numberwithin{equation}{section}
\theoremstyle{plain}
\newtheorem{theorem}[equation]{Theorem}
\newtheorem{prop}[equation]{Proposition}
\newtheorem{corollary}[equation]{Corollary}
\newtheorem{lemma}[equation]{Lemma}

\theoremstyle{definition}

\newtheorem{definition}[equation]{Definition}
\newtheorem{remark}[equation]{Remark}
\newtheorem{example}[equation]{Example}
\begin{document}

\title[Fine properties of BV mappings]{Fine properties of metric space-valued mappings of bounded variation in metric measure spaces} 
%\shorttitle[Fine properties of BV mappings]
\author[Caama\~{n}o, Kline, Shanmugalingam]{Iv\'an Caama\~{n}o, Josh Kline, Nageswari Shanmugalingam}
%\thanks{The research of J.K. and N.S. were partially suppored by the grant DMS~\#2054960 from
%the National Science Foundation (U.S.A.). J.K. was also partially suppored by the University of Cincinnati's
%University Research Council summer grant.}
\maketitle

\begin{abstract} {Here we consider 
two notions of mappings of bounded variation (BV) from the metric measure space into the metric space; 
one based
on relaxations of Newton-Sobolev functions, and the other based on a notion of AM-upper gradients. We show that when
the target metric space is a Banach space, these two notions coincide with comparable energies, but for more general
target metric spaces, the two notions can give different function-classes. We then consider the fine properties of BV mappings
(based on the AM-upper gradient property), and show that when the target space is a proper metric space, 
then for a BV mapping into the target space, 
co-dimension $1$-almost every point in the jump set of a BV mapping into the proper space has at least two, and at most
$k_0$, number of jump values associated with it, and that the preimage of balls around these jump values have lower density
at least $\gamma$ at that point. Here $k_0$ and $\gamma$ depend solely on the 
structural constants associated with the metric measure space,
and jump points are points at which the map is not approximately
continuous.}
\end{abstract}

\noindent
    {\small \emph{Key words and phrases}: {Bounded variation, metric measure space, Poincar\'e inequality, doubling measure,
vector-valued maps, approximate continuity, jump points, jump values}

\vskip .2cm

\noindent {\bf Acknowledgement:}The research of the authors J.K. and N.S. were partially 
supported by grant DMS~\#2054960 from
the National Science Foundation (U.S.A.). J.K. was also partially suppored by the University of Cincinnati's
University Research Council summer grant. 
I.C. was supported by the grant PID2022-138758NB-I00 (Spain).
The authors thank Panu Lahti for valuable feedback on early versions
of the manuscript.

\medskip


\section{Introduction}\label{Sec:1}

The theory of functions of bounded variation were first developed in order to study regularity properties of 
minimal surfaces, and a nice overview can be obtained from the collection~\cite{DeGi} and from the
discussion in~\cite{EvansGariepy}. Since then the theory has found applications in other areas as well,
including image processing~\cite{ACMM, CCN}, plasma physics~\cite{JK, GY}, and quasiconformal mappings~\cite{GK, K},
and the references contained in these papers provide further valuable information. 
In image processing or in plasma-blistering in media that is
not uniform and might even exhibit non-smoothness, a theory of functions of bounded variation in metric
spaces is useful. Recent research on mappings of finite distortion and quasisymmetric mappings 
indicate a need to understand metric space-valued mappings of bounded variation on metric measure 
spaces, see for instance~\cite{Ambrosio1990-II, Lah1, HKO, PR, CHM, BNP} for examples.
In this paper we seek to study mappings of bounded variation in non-smooth metric measure
spaces of controlled geometry, that is, spaces where the measure is doubling and supports a $1$-Poincar\'e inequality.

In comparison to Sobolev functions, functions of bounded variation exhibit less regularity; classic examples
include the Cantor staircase function on the Euclidean unit interval and characteristic functions of smooth Euclidean sets.
However, a Euclidean set whose characteristic function is of bounded variation can have non-smooth boundary.
As with characteristic functions of such sets, more general functions of bounded variation in Euclidean domains
exhibit discontinuity behavior along certain subsets, called jump sets. The situation gets more complicated when the
function of bounded variation is not real-valued but a map from a Euclidean domain into a metric space, as
in~\cite{Ambrosio1990-II}. Yet another layer of complication comes from considering functions of bounded 
variation from a metric measure space into a metric space. The goal of the present paper is to explore regularity properties
of such maps.

To do so, the first question to address is what is a reasonable notion of mappings of bounded variation from a metric 
measure space into a metric space. First proposed by Miranda Jr.~in~\cite{Miranda}, the notion of real-valued functions
on metric measure spaces equipped with a doubling measure supporting a $1$-Poincar\'e inequality have been extensively
studied, and the papers~\cite{Ambrosio, AMP, DEKS, EGLS, HKMM, HMM-1, HMM-2, HMM-3, 
Lah-Fed, LSh1, LZ} form a small sample of them. The papers~\cite{Miranda, Ambrosio, AMP, EGLS, LSh1} consider
the definition of functions of bounded variation in the metric setting via relaxation of Sobolev functions, while 
the papers~\cite{HKMM, HMM-1, HMM-2, HMM-3} consider functions of bounded variation as those whose local behavior
is controlled by a sequence of non-negative Borel functions that serve as a substitute for upper 
gradients~\cite{Martio}. In~\cite{DEKS} it was shown that these two approaches yield the same 
class of real-valued functions of bounded variation.

In the present paper we consider two definitions of mappings of bounded variation from a metric measure space into a
metric space or, in particular, a Banach space, by adopting the two approaches described above. We show that
when the domain metric measure space is complete, doubling, and supports a $1$-Poincar\'e inequality and the target metric
space is a Banach space, both notions yield the same class of maps. However, when the target metric space is not
a Banach space, the two approaches do not in general yield the same function class, with the notion of relaxation of 
Sobolev functions yielding a \emph{strictly smaller} subclass of maps. Thus, in the setting of general metric space target,
it is more appropriate to study mappings of bounded variation based on the sequence of upper gradients as first proposed
by Martio in~\cite{Martio}. Other alternate notions of 
metric-valued BV mappings defined via relaxation with simple maps and test plan-basd BV 
mappings was studied in \cite{BNP}, where it was shown to be equivalent to definitions given by 
test plans and post-composition with Lipschitz functions.  However, the fine properties of those mappings were not studied there.

Having made the choice of the definition of mappings of bounded variation, in the second part of the paper we 
explore the fine properties of mappings of bounded variation from a complete doubling metric measure space 
supporting a $1$-Poincar\'e inequality, into a proper metric space. We determine Hausdorff co-dimensional measure
properties of sets of jump discontinuity points of such mappings. The results about jump sets
and jump points for metric space-valued maps
are new even in Euclidean setting, thus augmenting the results found in~\cite{Ambrosio1990,
Ambrosio1990-II}.

The following are the two main results of this note. The first result focuses on comparing the two notions of 
mappings of bounded variation. The space $BV(X:V)$ is defined using the Miranda Jr.~\cite{Miranda} approach of relaxation
of Sobolev function class, while the space $BV_{AM}(X:V)$ is obtained by using sequences of non-negative Borel
functions that act as upper gradients as in~\cite{Martio}.

\begin{theorem}\label{thm:main}
Let $(X,d,\mu )$ be a complete doubling metric measure space supporting a $1$-Poincaré inequality, and
let $V$ be a Banach space. Suppose also that for $\mu$-a.e.~$x\in X$ we have that $\liminf_{r\to 0^+}\mu(B(x,r))/r=0$.
Then $BV(X:V)=BV_{AM}(X:V)$, with comparable BV energy seminorms.
\end{theorem}

Theorem~\ref{thm:main} will be proved in Section~\ref{Sec:4}. Before doing so, in Section~\ref{Sec:3} we adapt the notion
of Semmes pencil of curves and Poincar\'e inequality to the setting of Banach space-valued BV functions.

In the next main theorem, we determine the fine properties of a metric space-valued $BV_{AM}$-map, when the metric
space target is proper (that is, closed and bounded subsets of $Y$ are compact).

In what follows, we consider maps $u\in BV_{AM}(X:Y)$ with $(Y,d_Y)$ a metric space. A point $x\in X$
is said to be a \emph{point of approximate continuity} of $u$ if there is a point $y_x\in Y$ such that
for each $\eps>0$ we have
\[
\limsup_{r\to 0^+}\frac{\mu(B(x,r)\setminus u^{-1}(B(y_x,\eps)))}{\mu(B(x,r))}=0.
\]
We say that $x$ is a \emph{jump point} of $u$, that is, $x\in \mathcal{J}(u)$, if it is not a point of approximate continuity of $u$. 
If $x$ is a jump point of $u$, we say that a point $y\in Y$ is a \emph{jump value} of $u$ at $x$ if for all $\eps>0$,
we have
\[
\limsup_{r\to 0^+}\frac{\mu(B(x,r)\cap u^{-1}(B(y,\eps)))}{\mu(B(x,r))}>0.
\]

\begin{theorem}\label{thm:main2}
Let $(X,d,\mu)$ be a complete doubling metric measure space supporting a $1$-Poincar\'e inequality, and
let $(Y,d_Y)$ be a proper metric space. Then for each $u\in BV_{AM}(X:Y)$ there is a set $\mathcal{J}(u)\subset X$
such that $\mathcal{J}(u)$ is $\sigma$-finite with respect to the codimension~$1$ Hausdorff measure
$\mathcal{H}^{-1}$ on $X$ and a set $N\subset\mathcal{J}(u)$ with $\mathcal{H}^{-1}(N)=0$ such that
the following hold:
\begin{enumerate}
\item[(a)] Every point in $X\setminus\mathcal{J}(u)$ is a point of approximate continuity of $u$.
\item[(b)] For each $x_0\in\mathcal{J}(u)\setminus N$ there are at least two, and at most $k_0$, number of 
points $y_1,y_2,\cdots, y_k\in Y$ such that for each $\eps>0$ and $i=1, 2,\cdots, k$ we have 
\[
\liminf_{r\to 0^+}\frac{\mu(B(x,r)\cap u^{-1}(B(y_i,\eps)))}{\mu(B(x,r))}\ge \gamma,
\]
and
\[
\limsup_{r\to 0^+}\frac{\mu(B(x,r)\setminus \bigcup_{i=1}^ku^{-1}(B(y_i,\eps)))}{\mu(B(x,r))}=0.
\]
\end{enumerate}
In the above, both $k_0$ and $\gamma$ are constants that depend solely on the doubling and Poincar\'e constants of
the space $X$, and in particular are independent of $Y$, $u$ and $\eps$.
\end{theorem}

Theorem~\ref{thm:main2} will be proved in Section~\ref{sec:JumpSets}. The set $\mathcal{J}(u)$ is called the jump
set of $u$, and is constructed in the initial discussion of Section~\ref{sec:JumpSets} as the complement of the
set of points of approximate continuity of $u$, and so~(a) is immediate from the construction. The
$\sigma$-finiteness of the jump set is proved as Corollary~\ref{cor:SigmaFinite}. Subsequently,~(b) is
proved via Proposition~\ref{prop:finitejumps}, completing the proof.

\section{Background notions}\label{Sec:2}

In this note, $(X,d,\mu)$ will denote a metric measure space, where $(X,d)$ is a 
complete metric space and $\mu$ a Borel measure, and $V$ is a general Banach space. 
Balls centered at $x\in X$  with radius $r>0$ will be denoted $B(x,r)=\{y\in X\, :\, d(x,y)<r\}$.
A ball in $X$ may have more than one center and more than one radius. Hence, by a ball,
we understand that it comes with a pre-selected center and radius. The radius of a ball $B$ will
be denoted by $\rad(B)$. The \emph{closed} ball centered at $x$ with radius $r>0$ is the set
$\overline{B}(x,r):=\{z\in X\, :\, d(x,z)\le r\}$, and is in general potentially larger than the topological
closure of the open ball $B(x,r)$. Moreover, given two sets $E,F\subset X$, the distance between
them is denoted $\dist(E,F):=\inf\{d(x,y)\, :\, x\in E, y\in F\}$.

We will assume throughout that the measure $\mu$ is \emph{doubling}, that is, there is some
constant $C_d\ge 1$ such that whenever $x\in X$ and $r>0$, we have
\[
0<\mu(B(x,2r))\le C_d\, \mu(B(x,r))<\infty.
\]
Given such a measure $\mu$, and a set $A\subset X$, the \emph{co-dimension} $1$ Hausdorff measure
of $A$ is given by
\[
\mathcal{H}^{-1}(A):=\lim_{\delta\to 0^+}\, \inf\bigg\lbrace \sum_{i\in I\subset \N}\frac{\mu(B_i)}{\rad(B_i)}\, 
  :\, A\subset\bigcup_{i\in I}B_i,\, \rad(B_i)\le \delta\bigg\rbrace.
\]

Next, we introduce the definitions of two notions of mappings of bounded variation. The first one, $BV(X:V)$, 
was widely studied in \cite{Miranda}, 
while the other one, $BV_{AM}(X:V)$, was first introduced in \cite{Martio} and was proven to be equal 
to $BV(X:V)$ in \cite{DEKS}, when $V=\R$ and the measure on $X$ is doubling and supports a $1$-Poincar\'e inequality. 
As a natural question, we will study here the equality of both spaces when $V$ is a general Banach space.

\subsection{Vector-valued mappings of bounded variation via relaxation of Newton-Sobolev mappings}

Let $u\in L^1(X:V)$ with $L^1(X:V)$ in the sense of Bochner integrals, and define
\[
\Vert Du\Vert (X):=\inf\left\{ \liminf_{i\to\infty}\int_Xg_{u_i}\, d\mu:(u_i)_{i\in\N}\in N^{1,1}(X:V), u_i\overset{L^1}{\rightarrow}u\right\}.
\]

\begin{definition}\label{def:Miranda}
Let $(X,d,\mu )$ be a metric measure space and $V$ a Banach space. Following Miranda~\cite{Miranda}, we 
define $BV(X:V)$ to be the classes of mappings $u\in L^1(X:V)$ such that $\Vert Du\Vert (X)<\infty$. 
We denote $BV(X):=BV(X:\R )$.
\end{definition}

It was shown in~\cite{Miranda} that the map $U\mapsto \Vert Du\Vert(U)$ for open sets $U\subset X$ can be extended via
a Carath\'eodory construction to a Radon outer measure on $X$, which is also denoted by $\Vert Du\Vert$; in particular, for
Borel sets $A\subset X$ we set
\[
\Vert Du\Vert(A):=\inf\{\Vert Du\Vert(U)\, :\, U\text{ is open in }X \text{ and }A\subset U\}.
\]
If $E\subset X$ is a measurable set, we say that $E$ is of finite perimeter if $\chi_E\in BV(X)$. The perimeter measure
$P(E,\cdot):=\Vert D\chi_E\Vert(\cdot)$.
For functions in the class $BV(X)$ the following co-area formula is known.

\begin{lemma}\label{lem:coarea}(coarea formula, \emph{\cite[Proposition~4.2]{Miranda}})
Let $E\subset X$ Borel and $u\in BV (X)$. Then
\[
\Vert Du\Vert (E)=\int_{-\infty}^\infty P(\{ u>t\},E)dt.
\]
\end{lemma}

Thanks to the work of Ambrosio~\cite{Ambrosio}, we know the structure of sets of finite perimeter. To describe these results
we first describe the measure-theoretic and reduced boundaries of subsets of $X$. For $E\subset X$ we say that a point
$x\in X$ belongs to the measure-theoretic boundary $\partial_*E$ of $E$ if 
\begin{equation}\label{eq:MeasureTheoreticBoundary}
\limsup_{r\to 0^+}\frac{\mu(B(x,r)\cap E)}{\mu(B(x,r))}>0\ \text{ and }\ \limsup_{r\to 0^+}\frac{\mu(B(x,r)\setminus E)}{\mu(B(x,r))}>0.
\end{equation}
For a real number $\beta>0$ we say that $x\in X$ belongs to the reduced boundary $\Sigma_\beta E$ of $E$ if
\begin{equation}\label{eq:ReducedBoundary}
\liminf_{r\to 0^+}\frac{\mu(B(x,r)\cap E)}{\mu(B(x,r))}\ge\beta\ \text{ and }\ 
\liminf_{r\to 0^+}\frac{\mu(B(x,r)\setminus E)}{\mu(B(x,r))}\ge\beta.
\end{equation}

\begin{lemma}\label{lem:reducedBdy}
Suppose that $X$ is complete and that 
$\mu$ is doubling and supports a $1$-Poincar\'e inequality. Then there is a positive real number $\gamma\le 1/2$, 
depending only on the doubling constant and constants associated with the Poincar\'e inequality,
such that for each set $E$ of finite perimeter, 
\[
\mathcal{H}^{-1}(\partial_*E\setminus\Sigma_\gamma E)=0\ \text{ and }\
P(E,X)\approx \mathcal{H}^{-1}(\Sigma_\gamma E).
\]
\end{lemma}

We also point out that in fact, the property of a measurable set being of finite perimeter is characterized by
the property that $\mathcal{H}^{-1}(\Sigma_\gamma E)$ being finite; this result was first proved by Lahti~\cite{Lah-Fed},
and is new even in the Euclidean setting, refining Federer's characterization of Euclidean sets of finite perimeter.

\subsection{$1$-modulus and AM-modulus.}
Recall the definition of $1$-modulus of a family of non-constant, compact and rectifiable curves $\Gamma$:
\[
\Mod(\Gamma ):=\inf_\rho \int_X \rho\, d\mu 
\]
where the infimum is taken over all non negative Borel functions $\rho :X\to [0,+\infty ]$ s.t. $\int_\gamma \rho ds\geq 1$ for 
each $\gamma\in\Gamma$. It turns out that there is another notion of modulus that is better suited to the study of
BV functions. This notion, called $AM$-modulus, was first proposed by Martio in~\cite{Martio}.
Following~\cite{Martio} we define the $AM$-modulus to be
\[
\mathrm{AM}(\Gamma ):=\inf_{(\rho _i)_{i\in\N}}\liminf_{i\to\infty}\int_X\rho_id\mu,
\]
where the infimum is taken over all sequences of $\mathrm{AM}$-admisible functions, that is, sequences $(\rho _i)_i$ of non negative Borel functions such that for each $\gamma\in \Gamma$ we have
\[
\liminf_{i\to\infty }\int_\gamma \rho_i ds\geq 1.
\]
We say that a property holds for $1$-almost every curve (respectively AM-almost every curve) on $X$ if it holds outside a family of curves of zero $1$-modulus (resp.\ AM-modulus).\\
For a curve family $\Gamma $ we always have $\mathrm{AM}(\Gamma )\leq \mathrm{Mod}_1(\Gamma )$.

\begin{lemma}\label{lem:KosMac}
Let $\Gamma$ be a family of curves in $X$. Then 
\begin{enumerate}
\item[(a)] $\Mod(\Gamma)=0$ if and only if there is a non-negative Borel function $\rho\in L^1(X)$ such that for each
$\gamma\in\Gamma$ we have $\int_\gamma\rho\, ds=\infty$.
\item[(b)] $\AM(\Gamma)=0$ if and only if there is a sequence $(\rho_i)_{i\in\N}$ of non-negative Borel functions with 
$\sup_i\int_X\rho_i\, d\mu<\infty$ such that for each $\gamma\in\Gamma$ we have
\[
\liminf_{i\to\infty}\int_\gamma\rho_i\, ds=\infty.
\]
\end{enumerate}
\end{lemma}

\begin{proof}
A proof of~(a) can be found in~\cite[Lemma~5.2.8]{HKSTbook}. To prove~(b) we argue as follows. Suppose first that
$\AM(\Gamma)=0$. Then for each positive integer $k$ we can find a sequence $(\rho_{k,i})_{i\in\N}$ of non-negative
Borel functions on $X$ with $\liminf_{i\to\infty}\int_\gamma\rho_{k,i}\, ds\ge 1$ for each $\gamma\in\Gamma$, such that
$\sup_i\int_X\rho_{k,i}\, d\mu<2^{-k}$. For each positive integer $i$ we set $\rho_i=\sum_{k=1}^\infty\rho_{k,i}$.
By the monotone convergence theorem we know that for each $\gamma\in \Gamma$,
\[
\int_\gamma\rho_i\, ds=\sum_{k=1}^\infty\int_\gamma\rho_{k,i}\, ds,
\]
and so
\[
\liminf_k\int_\gamma\rho_i\, ds\ge \sum_{k=1}^\infty\liminf_k\int_\gamma\rho_{k,i}\, ds=\infty,
\]
and at the same time, for each positive integer $i$ we have
\[
\int_X\rho_i\, d\mu=\sum_{k=1}^\infty\int_X\rho_{k,i}\, d\mu\le \sum_{k=1}^\infty 2^{-k}=1.
\]
The desired conclusion follows.

Now suppose that $\Gamma$ is such that there is a sequence $(\rho_i)_{i\in\N}$ of non-negative Borel
functions on $X$ such that $\sup_i\int_X\rho_i\, d\mu=:\alpha<\infty$ and for each $\gamma\in\Gamma$
we have $\liminf_{i\to\infty}\int_\gamma\rho_i\, ds=\infty$. Then for each $\eps>0$ the sequence
$(\eps \rho_i)_{i\in\N}$ is a sequence of AM-admissible functions for $\Gamma$, with
$\limsup_i\int_X \eps\rho_i\, d\mu=\eps\alpha$. Thus 
$\AM(\Gamma)\le \eps\alpha$ for each $\eps>0$. Thus we have that $\AM(\Gamma)=0$.
\end{proof}

From the above lemma, it follows that if $\Gamma$ is a family of curves with 
$\mathrm{AM}(\Gamma)=0$, then for each $\varepsilon>0$ there is a
sequence $(\rho_i)_i$ such that $\sup_i\int_X\rho_i\, d\mu<\varepsilon$ and for each 
$\gamma\in\Gamma$ we have $\liminf_{i\to\infty}\int_\gamma\rho_i\, ds=\infty$.

\subsection{The notion of $BV_{AM}(X:V)$}

Now we turn our attention to the definition of $BV_{AM}(X:V)$. The notion of $BV_{AM}(X:\R)$ was first proposed by
Honzlov\'a-Exnerov\'a, Mal\'y, and Martio in a series of papers~\cite{HMM-1, HMM-2, HMM-3} using the notion of AM-modulus,
see also~\cite{HKMM}. This notion was adopted by Lahti in~\cite{Lah1} to study metric space-valued BV mappings.
In this section we focus on this notion of BV maps.

\begin{definition}
Let $(X,d,\mu )$ be a metric measure space and $V$ a Banach space. 
Let $(\rho_i)_{i\in\N}$ be a sequence of non-negative Borel functions on $X$. We say that this sequence is an
\emph{AM-bounding sequence} for a function $u:X\to V$ if 
for AM-a.e.~curve $\gamma:[a,b]\to X$ 
there is a null set $N_\gamma\subset[a,b]$ (that is, $\mathcal{H}^1(N_\gamma)=0$) such that for every 
$s,t\in [a,b]\setminus N_\gamma$ with $s<t$, we have 
\begin{equation}\label{eq:AMBV}
\Vert u(\gamma (s))-u(\gamma (t))\Vert\leq\liminf_{i\to\infty}\int_{\gamma|_{[s,t]}}\rho_i ds.
\end{equation}
We say that a mapping $u\in L^1(X:V)$ is in the class $BV_{AM}(X:V)$ if there is an AM-bounding sequence $(\rho_i)_{i\in\N}$ for
$u$ such that
\[
\liminf_{i\to\infty}\int_X\rho_i\, d\mu<\infty.
\] 
We set
\[
\Vert D_{AM}u\Vert (X):=\inf_{(\rho_i)_i}\liminf_{i\to\infty}\int_X\rho_i\,d\mu ,
\]
where the infimum is taken over all AM-bounding sequences $(\rho_i)_{i\in\N}$ of $u$.
\end{definition}

As in the case of the object $\Vert Du\Vert$, the above $\Vert D_{AM}u\Vert$ can be extended to be a Radon measure on $X$
via a Carath\'eodory construction, see for example~\cite{Martio} and Section~\ref{section:OuterMeasure}.

Note that in considering an AM-bounding sequence for $u$, we discount a family $\Gamma$ of curves in $X$ such that
$\AM(\Gamma)=0$. If the AM-bounding sequence $(\rho_i)_{i\in\N}$ is such that the exceptional family $\Gamma$ is 
empty, then we say that $(\rho_i)_{i\in\N}$ is a \emph{strong bounding sequence} for $u$.

\begin{lemma}\label{lem:null-mod}
Let $u\in BV_{AM}(X;V)$ and $v:X\to V$. 
Suppose that there is a set $N\subset X$ with $\mu(N)=0$ such that for
each $x\in X\setminus N$ we have $u(x)=v(x)$. Then a sequence $(\rho_i)_{i\in\N}$ is an AM-bounding sequence for
$u$ if and only if it is an AM-bounding sequence for $v$; hence $v\in BV_{AM}(X;V)$ with
$\Vert D_{AM}u\Vert(X)=\Vert D_{AM}v\Vert(X)$ and $\Vert D_{AM}(u-v)\Vert(X)=0$.
\end{lemma}

\begin{proof}
Since $N$ is a null-set, by enlarging it if need be, we can also assume that it is a Borel set as well. It follows that with
$\Gamma_N^+$ the collection of all non-constant compact rectifiable curves $\gamma:[a,b]\to X$ for which 
$\mathcal{H}^{1}(\gamma^{-1}(N))>0$, we have $\AM(\Gamma_N^+)=0$. Thus, for each AM-bounding sequence
$(\rho_i)_{i\in\N}$ of the original function $u$ and for each $\gamma\not\in\Gamma_0\cup\Gamma_N^+$,
we can replace $N_\gamma$ with $N_\gamma\cup\gamma^{-1}(N)$ to see that this is an AM-bounding sequence for
$v$ as well.

Since $u-v=0$ $\mu$-a.e.~in $X$, the final claims follows from noting that $\AM(\Gamma_N^+)=0$ and so
the sequence $(g_i)_{i\in\N}$, with each $g_i$ the zero function, is an AM-bounding sequence for $u-v$.
\end{proof}

\begin{lemma}\label{rem:AllCurves}
Suppose that $(\rho_i)_{i\in\N}$ is an AM-bounding sequence for a function $u$ on $X$ such that
$\sup_i\int_X\rho_i\, d\mu=:\tau<\infty$. Then for each $\eps>0$ we can find a strong bounding sequence $(g_i)_{i\in\N}$
of $u$ such that for each $i\in\N$ we have $\int_X|g_i-\rho_i|\, d\mu<\eps$.
\end{lemma}

\begin{proof}
Let $u\in BV_{AM}(X:V)$ and $(\rho_i )_{i\in\N}$. Then there exists $\Gamma$ with AM$(\Gamma)=0$ such that for 
each non constant compact rectifiable curve $\gamma\notin\Gamma$, the relation~\eqref{eq:AMBV} holds for 
$s,t\in \mathrm{dom}(\gamma )\backslash \gamma^{-1}(N_\gamma )$ where $\mathcal{H}^1(N_\gamma )=0$. Since 
AM$(\Gamma)=0$, by Lemma~\ref{lem:KosMac} there exists a sequence of nonnegative Borel functions $(g_i)_i$ such that
\[
\liminf_{i\to\infty}\int_Xg_i\,d\mu <\infty\quad\text{and}\quad \liminf_{i\to\infty}\int_\gamma g_i\,ds =\infty\;\;\forall\gamma\in\Gamma.
\]
Now let $\Gamma_0$ be the family of all nonconstant compact rectifiable curves $\gamma$ in $X$ for which we have
$\liminf_{i\to\infty}\int_\gamma g_ids =\infty$. Then AM$(\Gamma_0)=0$ and each subcurve of a curve that is not
in $\Gamma_0$ is 
also not in $\Gamma_0$. For each $\varepsilon>0$, since $\Gamma\subset\Gamma_0$, we have that for each 
$\gamma\notin\Gamma_0$,
\begin{equation}\label{eq:bound2}
\Vert u(\gamma (s))-u(\gamma (t))\Vert
\leq\liminf_{i\to\infty}\int_{\gamma |_{[s,t]}}\rho_i ds \leq\liminf_{i\to\infty}\int_{\gamma |_{[s,t]}}\rho_i+\varepsilon g_i ds
\end{equation}
for every $s,t\notin\gamma^{-1}(N_\gamma)$. 

If $\gamma\in\Gamma_0$ such that every subcurve of $\gamma$ also belongs to $\Gamma_0$, then for each
$s,t\in [a,b]$ with $s<t$ we have $\liminf_{j\to\infty}\int_{\gamma\vert_{[s,t]}}g_i\, ds=\infty$, and so the choice of
$N_\gamma=\emptyset$ works. If it is not the case that every subcurve of $\gamma$ also belongs to $\Gamma_0$, 
then let $\mathcal{C}_0(\gamma)$ be the collection of all non-degenerate (that is, containing more than one point)
intervals $I\subset[a,b]$ for which, whenever $J$ is a compact subinterval of $I$ we must have 
$\gamma\vert_J\not\in\Gamma_0$, and whenever $J$ is a compact subinterval of $[a,b]$ containing $I$ in its interior,
we must have $\gamma\vert_J\in\Gamma_0$. 
By the maximality of the intervals in the collection $\mathcal{C}_0(\gamma)$, two intervals in this collection are either disjoint
or are equal as intervals. Moreover, these intervals have non-empty interior. It follows that as $\Q$ is dense in $\R$,
the collection $\mathcal{C}_0(\gamma)$ is countable.

With $\gamma:[a,b]\to X$, consider all $a_0,b_0\in[a,b]\cap\Q$ with $a_0<b_0$ for which 
$\gamma\vert_{[a_0,b_0]}\not\in\Gamma_0$, that is,
$\liminf_{i\to\infty}\int_{\gamma\vert_{[a_0,b_0]}}g_i\, ds<\infty$; hence there is a corresponding null set
$N[a_0,b_0]\subset[a_0,b_0]$ with $\mathcal{H}^1(N[a_0,b_0])=0$, so that for each $s,t\in[a_0,b_0]\setminus N[a_0,b_0]$
we have~\eqref{eq:bound2} holding true. 
Let $\mathcal{C}(\gamma)$ be the collection of all such $[a_0,b_0]\subset[a,b]$, and
set 
\[
N_\gamma:=\bigcup_{[a_0,b_0]\in\mathcal{C}(\gamma)}N[a_0,b_0]\cup
  \bigcup_{J\in\mathcal{C}_0(\gamma)}\{\inf J, \sup J\}.
\] 
Note that as $a_0,b_0\in\Q$, the collection
$\mathcal{C}(\gamma)$ is a countable collection. Hence $\mathcal{H}^1(N_\gamma)=0$ by the subadditivity of $\mathcal{H}^1$
on $[a,b]$. Now let $s,t\in[a,b]\setminus N_\gamma$ with $s<t$. If $[s,t]\subset[a_0,b_0]$ for some 
$[a_0,b_0]\in\mathcal{C}(\gamma)$, then~\eqref{eq:bound2} holds. If there is no $[a_0,b_0]\in\mathcal{C}(\gamma)$ for which
$[s,t]\subset[a_0,b_0]$, then we must have that $\liminf_{i\to\infty}\int_{\gamma\vert_{[s,t]}}g_i\, ds=\infty$, and so we now have
\[
\Vert u(\gamma (s))-u(\gamma (t))\Vert \le \liminf_{i\to\infty}\int_{\gamma\vert_{[s,t]}}(\rho_i+\eps g_i)\, ds=\infty.
\] 
Therefore $(\rho_i+\varepsilon g_i)_i$ satisfies \eqref{eq:AMBV} for every non constant compact rectifiable curve. 
Moreover, since $\varepsilon$ is arbitrary, one can approach the energy $\Vert D_{AM}u\Vert (X)$ just by 
taking the infimum over the upper bounds of $u$ that verify \eqref{eq:AMBV} for every non constant compact rectifiable curve.
\end{proof}

\begin{lemma}\label{lem:Leibnitz}
Suppose that $(\rho_i)_i$ is an AM--bounding sequence for $u$ and that $\eta$ is a non-negative $L$-Lipschitz function
with support in a bounded set $U$; moreover, suppose that $\eta$ is constant on an open set $V\Subset U$. Then
$(\eta \rho_i+L\, |u|\, \chi_{U\setminus V})_i$ is an AM--bounding sequence for $\eta u$.
\end{lemma}

\begin{proof}
For each $i\in\N$ we set $g_i:=\eta \rho_i+L\, |u|\, \chi_{U\setminus V}$.

Let $\Gamma$ be the exceptional set for the AM--bounding sequence $(\rho_i)_i$ with respect to $u$; so 
$\AM(\Gamma)=0$, and for each non-constant compact rectifiable curve $\gamma:[a,b]\to X$ with $\gamma\not\in\Gamma$,
we have a null set $N_\gamma\subset [a,b]$ with $\mathcal{H}^1(N_\gamma)=0$ such that whenever
$t,s\in[a,b]\setminus N_\gamma$ with $t<s$, we have
\[
|u(\gamma(t))-u(\gamma(s))|\le \liminf_{i\to\infty}\int_{\gamma\vert_{[t,s]}}\rho_i\, ds.
\]
Fix such $s,t\in[a,b]$. 
Consider a partition $t=t_0<t_1<\cdots<t_k=s$ of the interval $[t,s]$ so that $t_1,\cdots, t_{k-1}\not\in N_\gamma$.
Then
\begin{align*}
|u(\gamma(t))\eta(\gamma(t))-u(\gamma(s))\eta(\gamma(s))|
&\le\sum_{j=1}^k|u(\gamma(t_{j-1}))\eta(\gamma(t_{j-1}))-u(\gamma(t_{j}))\eta(\gamma(t_{j}))|\\
&\le \sum_{j=1}^k|u(\gamma(t_{j-1}))\eta(\gamma(t_{j-1}))-u(\gamma(t_{j}))\eta(\gamma(t_{j-1}))|\\
 &\hskip 2cm+\sum_{j=1}^k|u(\gamma(t_{j}))\eta(\gamma(t_{j-1}))-u(\gamma(t_{j}))\eta(\gamma(t_{j}))|\\
&\le \liminf_{i\to\infty}\sum_{j=1}^k\int_{\gamma\vert_{[t_{j-1},t_j]}}\left[\eta(\gamma(t_{j-1}))\rho_i\, 
+|u(\gamma(t_j))|\text{Lip}\,\eta\, \right]\, ds.
\end{align*}
The above must be true for all such partitions of the interval $[t,s]$.
Since $u\circ\gamma$ and $\text{Lip}\,\eta\circ\gamma$ are Borel functions on $[a,b]$, it follows that
\[
|u(\gamma(t))\eta(\gamma(t))-u(\gamma(s))\eta(\gamma(s))|
\le\liminf_{i\to\infty}\int_{\gamma\vert_{[t,s]}}g_i\, ds.\qedhere
\]
\end{proof}

\begin{lemma}\label{lem:BVinBVAM}
Let $u\in BV(X:V)$.  Then the upper gradients of an approximating sequence form an AM-bounding sequence for $u$. 
In particular $u\in BV_{AM}(X:V)$ with $\Vert D_{AM}u\Vert(X)\le \Vert Du\Vert(X)$. 
\end{lemma}

\begin{proof}
If $u\in BV(X:V)$, then there exists a sequence $(u_i)_{i\in\N}\in N^{1,1}(X:V)$ with upper gradients $g_i$ such that 
\[
\limsup_{i\to\infty}\int_X g_id\mu <\infty,
\]
and such that $u_i\to u$ in $L^1(X:V)$. In particular, (passing to a subsequence 
if necessary) $u_i\to u$ pointwise almost everywhere in $X$. Then there exists a null set 
$N$ such that $\lim_{i\to\infty} u_i(x)=u(x)$ for every $x\in X\setminus N$. By enlarging $N$ if necessary, we may also
assume that $N$ is a Borel set (recall that $\mu$ is Borel regular).
Hence, by considering the non-negative Borel measurable function $\rho=\infty\chi_N$ on $X$, we know 
that if $\Gamma_N^+$ is the collection of all compact non-constant rectifiable curves $\gamma$ in $X$
with $\mathcal{H}^1(\gamma^{-1}(N))>0$, then $\AM(\Gamma_N^+)\le \Mod(\Gamma_N^+)=0$.

Let $\gamma:[a,b]\to X$ be a 
non-constant compact rectifiable curve such that $\gamma\not\in\Gamma_N^+$. We set $N_\gamma:=\gamma^{-1}(N)$.
Then for $s,t\in[a,b]\setminus N_\gamma$ with $s<t$, we have that
\[
\lim_{i\to\infty}u_i(\gamma(t))-u_i(\gamma(s))=u(\gamma(t))-u(\gamma(s)).
\]
Moreover, since $g_i$ is an upper gradient of $u_i$, it follows that
\[
\Vert u_i(\gamma(t))-u_i(\gamma(s))\Vert\le \int_{\gamma\vert_{[s,t]}}g_i\, ds.
\]
Combining the above two, we see that 
for $s,t\in[a,b]\setminus N_\gamma$ with $s<t$, we have
\[
\Vert u(\gamma(t))-u(\gamma(s))\Vert\le \liminf_{i\to\infty}\int_{\gamma\vert_{[s,t]}}g_i\, ds.
\] 
Thus we have shown that $(g_i)_{i\in\N}$ is an AM-bounding sequence
for $u$. Now fix $\varepsilon >0$ and choose $(u_i)_{i\in\N}$ and $(g_i)_{i\in\N}$ such that 
$$\liminf_{i\to\infty}\int_Xg_id\mu\leq \Vert Du\Vert (X)+\varepsilon .$$
Since the previous argument holds for any choice of $(u_i)_{i\in\N}$ and $(g_i)_{i\in\N}$, we have that $(g_i)_{i\in\N}$ 
is an AM-bounding sequence for $u$, and hence,
\[
\Vert D_\mathrm{AM}u\Vert(X)\leq \liminf_{i\to\infty}\int_Xg_id\mu \leq \Vert Du\Vert (X) +\varepsilon,
\]
and then taking $\varepsilon\to 0$ completes the proof.
\end{proof}

\begin{remark}
Notice that condition \eqref{eq:AMBV} holds outside a null set $N_\gamma$, which
depends on the choice of $\gamma$, but as 
seen in the previous proof, whenever $u\in BV(X:V)$, we can choose a null set $N\subset X$ to be 
independant of the curve, such that $N_\gamma=\gamma^{-1}(N)$. 
In the following sections we will prove that $BV(X:V)=BV_{AM}(X:V)$.
However the construction of the 
approximation by Newtonian mappings of a $BV_{AM}$-mapping will yield a sequence of upper gradients 
different to the original AM-bounding sequence for $u$, and so in general we can't assume that \emph{every} 
AM-bounding sequence of $u$ comes with a universal null set $N\subset X$ as above. 
\end{remark}

\subsection{Metric space-valued functions of bounded variation}\label{Sec:MetricBV}

In the prior two sections we considered mappings from the metric space $X$ into a Banach space $V$; in this section we 
consider the case of mappings into a metric space $(Y,d_Y)$. We will assume here that $Y$ is complete and 
separable.  Given $\Omega\subset X$  and $y_0\in Y$, we mean by $f\in L^1(\Omega:Y,y_0)$
that 
\[
\int_\Omega d_Y(f(\cdot),y_0)d\mu<\infty.
\]
We say that $f\in L^1_{loc}(X:Y)$ if for each $y_0\in Y$ (or, equivalently, for some $y_0\in Y$)
we have that the real-valued function $x\mapsto d_Y(f(x),y_0)$ belongs to $L^1_{loc}(X)$.

We begin with the intrinsic definitions, analogous to the definitions of $BV(X:V)$ and $BV_{AM}(X:V)$ given above. 

\begin{lemma}\label{lem:BV-AM-metric}
Let $u:X\to Y$ be a measurable function. Then the following are equivalent:
\begin{enumerate}
\item[(a)] There is a family $\Gamma_0$ of curves in $X$ with $\AM(\Gamma_0)=0$ and a sequence $(\rho_i)_{i\in\N}$
of non-negative Borel measurable functions on $X$ such that for each non-constant compact rectifiable curve 
$\gamma:[a,b]\to X$ with $\gamma\not\in\Gamma_0$ there is a set $N_\gamma\subset[a,b]$ with 
$\mathcal{H}^{-1}(N_\gamma)=0$ such that for each $s,t\in[a,b]\setminus N_\gamma$ with $s<t$ we have
\[
d_Y(u(\gamma(s)),u(\gamma(t)))\le \liminf_{i\to\infty}\int_{\gamma\vert_{[s,t]}}\rho_i\, ds,
\]
with 
\[
\sup_i\int_X\rho_i\, d\mu<\infty.
\]
\item[(b)] For every Banach space $V$ and isometric embedding $\Phi:Y\to V$ we have that $\Phi\circ u\in BV_{AM}(X;V)$.
\item[(c)] There is a Banach space $V$ and an isometric embedding 
$\Phi:Y\to V$ such that $\Phi\circ u\in BV_{AM}(X;V)$.
\end{enumerate}
In any (and hence all) of the cases above, we have that
\[
\Vert D_{AM}\Phi\circ u\Vert(X)=\inf_{(\rho_i)_{i\in\N}}\liminf_{i\to\infty} \int_X\rho_i\, d\mu
\]
where the infimum is over all sequences $(\rho_i)_{i\in\N}$ satisfying~(a) above.
\end{lemma}

\begin{proof}
If $V$ is a Banach space and $\Phi$ is an isometric embedding of $Y$ into $V$, then for each $x,z\in X$ we have that
$d_Y(u(x),u(z))=\Vert \Phi(u(x))-\Phi(u(z))\Vert$, and so we know that (a) implies (b) and that (b) implies (c).
Indeed, every complete separable metric space can be isometrically embedded in the Banach space $\ell^\infty$.

Thus it only remains to show that (c) implies (a). To this end, suppose that $V$ is a Banach space, $\Phi$ an isometric
embedding of $Y$ into $V$, and that $\Phi\circ u\in BV_{AM}(X;V)$. Then there is a sequence $(\rho_i)_{i\in\N}$
that is an AM-bounding sequence for $\Phi\circ u$, and a family $\Gamma_0$ with $\AM(\Gamma_0)=0$ such that
whenever
$\gamma:[a,b]\to X$ with $\gamma\not\in\Gamma_0$ there is a set $N_\gamma\subset[a,b]$ with 
$\mathcal{H}^{-1}(N_\gamma)=0$ such that for each $s,t\in[a,b]\setminus N_\gamma$ with $s<t$ we have
\[
\Vert\Phi\circ u(\gamma(s))-\Phi\circ u(\gamma(t))\Vert\le \liminf_{i\to\infty}\int_{\gamma\vert_{[s,t]}}\rho_i\, ds,
\]
with 
\[
\sup_i\int_X\rho_i\, d\mu<\infty.
\]
As $\Phi$ is an isometric embedding of $Y$ into $V$, it follows that 
$d_Y(u(\gamma(s)), u(\gamma(t)))=\Vert\Phi\circ u(\gamma(s))-\Phi\circ u(\gamma(t))\Vert$. The condition~(a) follows.

The above argument shows that the sequence $(\rho_i)_{i\in\N}$ satisfies~(a) if and only if it is an AM-bounding sequence
for $\Phi\circ u$. The final claim of the above lemma now follows.
\end{proof}

\begin{definition}
We say that a map $u:X\to Y$ is in $BV_{AM}(X;Y)$ if $u\in L^1(X:Y)$ and there is a sequence $(\rho_i)_{i\in\N}$ satisfying
the hypothesis of Lemma~\ref{lem:BV-AM-metric}~(a).
\end{definition}

\begin{lemma} \label{lem-isom}
Let $\Phi:Y\to V$ be an isometric embedding of the metric space $Y$ into a Banach space $V$. Suppose that
\begin{enumerate}
\item[(a)] $u\in BV_{AM}(X;V)$, and
\item[(b)] there is a set $N\subset X$ with $\mu(N)=0$ such that for each $x\in X\setminus N$ we have that
$u(x)\in \Phi(Y)$.
\end{enumerate}
Then $u\circ\Phi^{-1}\in BV_{AM}(X;Y)$. Here $\Phi^{-1}$ stands in for the inverse map of the bijective map
$\Phi:Y\to\Phi(Y)$.
\end{lemma}

\begin{proof}
Since any modification of $u$ on a set of $\mu$-measure zero results in the same equivalence class of $u$ in
$BV_{AM}(X;V)$ (see Lemma~\ref{lem:null-mod}), we can modify $u$ on $N$ by setting $u(x)$ to be some fixed
point in $\Phi(Y)$ if $x\in N$. The conclusion now follows from the previous lemma.
\end{proof}

Unlike $BV_{AM}(X;Y)$, the situation for $BV(X; Y)$ is more complicated. 

\begin{definition}
We say that a map $u:X\to Y$ is in $BV(X;Y)$ if there is a sequence $(u_k)_{k\in\N}$ from 
$N^{1,1}(X;Y)$ such that $u_k\to u$ in $L^1(X;Y)$ and $\sup_{k\in\N}\int_Xg_{u_k}\, d\mu<\infty$.
Here $g_{u_k}$ is the minimal $1$-weak upper gradient of $u_k$ in the sense of~\cite[page~161]{HKSTbook}.
\end{definition}

We will see that $BV_{AM}(X;V)=BV(X;V)$ whenever $V$ is a Banach space (see
Theorem~\ref{thm:main}), and as in the proof of Lemma~\ref{lem:BVinBVAM}, we can see that
$BV(X;Y)\subset BV_{AM}(X;Y)$.
However, $BV(X;Y)$ is in general a
strictly smaller subset of $BV_{AM}(X;Y)$.
This supports the choice of $BV_{AM}(X;Y)$ as the space of mappings of bounded variation in~\cite{Lah1, Lah2}.

\begin{example}
Consider $X=[-1,1]$ and $Y=\{ 0,1\}$. Let $u:=\chi _{[0,1]}$ and $\rho_i:=i\chi_{[-1/i,0]}$. Then for each $x,y\in [-1,1]$, if 
$x,y<0$ or $x,y\geq 0$ then $|u(x)-u(y)|=0$ so it is immediate that $(\rho_i)_i$ satisfies the upper bound inequality. 
If $x<0$ and $y\ge 0$ then
\[
\liminf_{i\to\infty}\int_x^y i\chi_{[-1/i,0]} d\mathcal{L}^1=
1=|u(x)-u(y)|.
\]
Thus the sequence $(\rho_i)_{i\in\N}$ is an AM-bounding sequence for $u$ (and indeed, it is a strong
bounding sequence for $u$.
Therefore we know that $u\in BV_{AM}(X:Y)$. However Newtonian mappings 
are absolutely continuous on $[-1,1]$, and so, since $Y=\{ 0,1\}$, they must be constant; hence it is not possible 
to approximate (in $L^1$ norm) $u$ by Newtonian mappings, proving that $u\notin BV(X:Y)$.
\end{example}

While the above example shows how topological obstructions can prevent approximations by $N^{1,1}$-maps, the next example
provides a more analytical obstruction.

\begin{example}
Let $X=[-1,1]$ and $Y=(\{0\}\times[0,1])\cup([0,1]\times\{1\})\cup(\{1\}\times[0,1])$ be both equipped with the Euclidean metric,
and $X$ be also equipped with the Lebesgue measure $\mathcal{L}^1$. Let $u\in BV_{AM}(X:Y)$ be given by 
$u(x)=(0,0)$ when $-1\le x\le 0$ and $u(x)=(1,0)$ when $0<x\le 1$. If $(u_k)_k$ is a sequence of functions from
$N^{1,1}(X:Y)$ such that $u_k\to u$ in $L^1(X)$, then for sufficiently large $k$ we know that are points $x_k,y_k\in X$ 
with $|x_k+1|<\tfrac{1}{10}$ and $|y_k-1|<\tfrac{1}{10}$ such that $|u(x_k)-(0,0)|<\tfrac{1}{10}$ and $|u(y_k)-(1,0)|<\tfrac{1}{10}$. 
On the other hand, by the
absolute continuity on $[-1,1]$ of functions in $N^{1,1}(X:Y)$, we have that 
$\int_Xg_{u_k}\, d\mathcal{L}^1=\text{length}(u_k)\ge 1+2\times\, \tfrac{9}{10}=\tfrac{14}{5}$. In inferring the above, we used
the fact that $u_k$ is also a path in $Y$. On the other hand, $\Vert D_{AM}u\Vert(X)=1$, and so it is not possible to
have $\Vert D_{AM}u\Vert(X)=\inf_{(u_k)\subset N^{1,1}(X:Y)}\, \liminf_{k\to\infty}\int_Xg_{u_k}\, d\mathcal{L}^1$.
\end{example}

The above example does not preclude an $L^1$-approximation of the AM--BV function by a sequence of 
Newton-Sobolev functions. The next example gives a metric obstruction to the existence of even such an approximation.

\begin{example}
Let $X=[-1,1]$ and $Y=(\{0\}\times[-1,1])\cup\{(x,\sin(1/x))\, :\, 0<x\le 1\}$ be both equipped with the Euclidean metrics, and
$X$ also be equipped with the $1$-dimensional Lebesgue measure. Let $u:X\to Y$ be given by
$u(x)=(0,0)$ if $-1\le x\le 0$, and $u(x)=(1,\sin(1))$ when $0<x\le 1$. Then 
$\Vert D_{AM}u\Vert(X)=\sqrt{1+\sin^2(1)}$, but due to the absolute continuity of functions in $N^{1,1}(X:Y)$ and the 
lack of paths in $Y$ connecting $u(-1)$ to $u(1)$, there can be no sequence of functions in $N^{1,1}(X:Y)$ that 
gives an $L^1$-approximation of $u$.
\end{example}

By the ACL (absolute continuity on lines) property of functions in $N^{1,1}(X:Y)$, the above examples have higher
dimensional analogs, but we will not go into details here.

\begin{remark}
While $BV(X:Y)$ need not equal $BV_{AM}(X:Y)$ in general, we do have a relationship between the two notions. 
Thanks to the Kuratowski embedding theorem (\cite[page~100]{HKSTbook}), 
we can isometrically embed any separable metric space $(Y,d_Y)$ into
the Banach space $\ell^\infty$. With $V$ any Banach space and $\Phi:Y\to V$ an isometric embedding, whenever
$Y$ is complete, thanks to Theorem~\ref{thm:main} and Lemma~\ref{lem-isom}, we know that $BV_{AM}(X:Y)$ is 
the same as the class
\[
\Bigg\lbrace u\in BV(X:V)\, :\, \mu(\{x\in X\, :\, u(x)\not\in \Phi(Y)\})=0\Bigg\rbrace.
\]
\end{remark}

\section{Poincaré inequalities and Semmes Pencil.}\label{Sec:3}
We say that the metric measure space $(X,d,\mu)$ supports a $1$-Poincaré inequality if there are 
constants $C>0,\lambda \geq 1$ such that for each 
$u,g\in L^1_{loc}(X)$, with $g$ an upper gradient of $u$, we have
\[
\vint{B}\Vert u-u_B\Vert d\mu\leq C\mathrm{rad}(B)\vint{\lambda B}gd\mu 
\]
for each ball $B\subset X$.  The metric measure space $X$ supports an AM--Poincaré inequality if there are
$C>0,\,\lambda\geq 1$ so that for each $u\in BV_{AM}(X:V)$ and any AM-upper bound $(\rho_i)_i$ of $u$ we have
\[
\vint{B}\Vert u-u_B\Vert d\mu\leq C\mathrm{rad}(B)\liminf_{i\to\infty }\vint{\lambda B}\rho_i d\mu 
\]
for each ball $B\subset X$.  As with the $1$-Poincar\'e inequality, the AM-Poincar\'e inequality 
implies the following version, which involves the AM-BV energy:

\begin{lemma}\label{eq:PI-BV-AM}
If $X$ supports an AM--Poincar\'e inequality, then for $u\in BV_{AM}(X:V),$ we have that 
\[
\vint{B}\|u-u_B\|d\mu\le C\rad(B)\frac{\|D_{AM}u\|(\lambda B)}{\mu(\lambda B)}
\]
for each ball $B\subset X$.
\end{lemma}

\begin{proof}
Let $\eps>0$ and let $\{\rho_k\}_k$ be an AM-upper bound for $u$ in $\lambda B$ such that 
\[
\liminf_{k\to\infty}\int_{\lambda B}\rho_kd\mu<\|D_{AM}u\|(\lambda B)+\eps.
\]
Let $0<\delta<1/2$,
and let $\eta$ be a $1/\delta$-Lipschitz function such that $\eta=1$ on 
$\lambda(1-\delta)B$ and $\eta=0$ on $X\setminus\lambda B$, with $0\le \eta\le 1$ on $X$.  We then have that 
$\{\eta\rho_k+\delta^{-1}|u|\,\chi_{\lambda B\setminus\lambda(1-\delta)B}\}_k$ is an AM-upper bound for $\eta u$ in $X$,
see for example Lemma~\ref{lem:Leibnitz} above.
Thus, by the AM--Poincar\'e inequality and by the doubling property of $\mu$, we have that 
\begin{align*}
\vint{(1-\delta)B}\vint{(1-\delta)B}\|u(y)-u(x)&\|\, d\mu(y)\, d\mu(x)\le 2\vint{(1-\delta)B}\|u-u_{(1-\delta)B}\|d\mu\\
&\le C(1-\delta)\rad(B)\liminf_{k\to\infty}\vint{\lambda(1-\delta)B}\left(\eta\rho_k
  +\frac{|u|}{\delta}\chi_{\lambda B\setminus\lambda(1-\delta)B}\right)d\mu\\
	&\le C\rad(B)\liminf_{k\to\infty}\vint{\lambda B}\rho_kd\mu\\
	&<C\rad(B)\left(\frac{\|D_{AM}u\|(\lambda B)+\eps}{\mu(\lambda B)}\right).
\end{align*}
Now letting $\delta\to 0$ and taking $\eps\to 0,$ we have that 
\[
\vint{B}\|u-u_B\|d\mu\le C\rad(B)\frac{\|D_{AM}u\|(\lambda B)}{\mu(\lambda B)}. 
\]
Here $u_B$ is the Bochner integral average of $u$ over the ball $B$.
\end{proof} 

The goal of this section is to see that if $X$ supports a $1$-Poincaré inequality then it supports an AM--Poincar\'e 
inequality. For that we will use the fact that spaces supporting a $1$-Poincaré inequality have a Semmes 
Pencil of curves (see~\cite[Theorem~3.10]{DEKS}).

\begin{definition}\label{def:SemmesPencil}
We say that the metric measure space $(X,d,\mu)$ supports a Semmes pencil of curves if there exists $C>0$ 
so that for each $x,y\in X$ there exists a 
family  $\Gamma_{x,y}$ of non-constant compact rectifiable curves equipped with a probability measure 
$\sigma_{x,y}$ such that each $\gamma\in\Gamma_{x,y}$ connects $x$ to $y$, $\ell (\gamma )\leq Cd(x,y)$, and for each 
Borel set $A\subset X$ the map $\gamma\mapsto \ell(\gamma\cap A)$ is $\sigma_{x,y}$-measurable with
\[
\int_{\Gamma_{x,y}} \ell (\gamma\cap A) d\sigma_{x,y} (\gamma )\leq C\int_{A\cap CB_{x,y}}R_{x,y}(z)d\mu (z),
\]
where $CB_{x,y}:=B(x,Cd(x,y))\cup B(y,Cd(x,y))$ and
\[
R_{x,y}(z):=\frac{d(x,z)}{\mu (B(x,d(x,z))}+\frac{d(y,z)}{\mu (B(y,d(y,z))}.
\]
\end{definition}

\begin{theorem}\label{thm:semmesBVAM}
Suppose that $\mu$ is doubling, $X$ has a Semmes pencil of curves, and that for $\mu$-a.e.~$x\in X$
we have $\liminf_{r\to 0^+}\mu(B(x,r))/r=0$.
Then $X$ supports an AM--Poincar\'e inequality 
for every Banach space $V$.
\end{theorem}

The proof follows along the lines of~\cite[Proposition~3.9]{DEKS}, but as we are now dealing with a vector-valued map,
we provide the complete proof here, especially since there seems to be a gap in the details of the proof in~\cite{DEKS} 
which we fixed here. In doing so, we saw that we needed the additional condition that $\liminf_{r\to 0^+}\tfrac{\mu(B(x,r))}{r}=0$
for almost every $x\in X$. This condition fails for example when $X=\R$, but in spaces that are not one-dimensional in nature
this is automatically satisfied via the upper mass bound estimates for the doubling measure $\mu$ on the connected space $X$,
when the upper mass bound exponent can be taken to be larger than $1$. When $X=\R^2$ is equipped with the
measure $d\mu(x)=|x|^{-1}d\mathcal{L}^2$, the point $x=0$ fails the condition $\liminf_{r\to 0^+}\tfrac{\mu(B(x,r))}{r}=0$, but
this condition holds at all other $x\in \R^2$.

\begin{proof}
Fix a Banach space $V$ and let $u\in BV_{AM}(X:V)$.  Let $B$ be a ball in $X$ and
$(\rho_k)_{k=1}^\infty$ be a strong bounding sequence 
for $u$ in $4C B$ with $C$ the constant associated with the Semmes pencil of curves, that is,
condition~\eqref{eq:AMBV} holds for all curves in $4C B$. By the Lebesgue differentiation theorem 
for vector-valued functions (see for instance~\cite[page~77]{HKSTbook}),
we know that 
$\mu (N)=0$, where $N$ is the set of points $x\in X$ for which either $\liminf_{r\to 0^+}\mu(B(x,r))/r>0$ or
\[
\limsup_{r\to 0^+}\vint{B(x,r)}\Vert u(x)-u(z)\Vert d\mu (z)>0.
\]
Let $x,y\in X\setminus N$ be two distinct points, and for each $\varepsilon >0$ consider the sets
\[
E_\varepsilon (x):=\{ z\in X:\Vert u(x)-u(z)\Vert >\varepsilon \},\quad E_\varepsilon (y):=\{ z\in X:\Vert u(y)-u(z)\Vert >\varepsilon \}.
\]
Now, since $x$ and $y$ are Lebesgue points of $u$, we know that
\[
\limsup_{r\to 0}\frac{\mu (B(x,r)\cap E_\varepsilon (x))}{\mu (B(x,r))}=0
\quad\text{and}\quad \limsup_{r\to 0}\frac{\mu (B(y,r)\cap E_\varepsilon (y))}{\mu (B(y,r))}=0.
\]
Let $\{ r_i\}_i$ be a decreasing sequence of radii so that $r_1\leq \frac{1}{4}d(x,y)$, $r_{i+1}\leq \frac{1}{4}r_i$, 
$\frac{\mu (B(x,r_i))}{r_i}\leq 2^{-i}$, $\frac{\mu (B(y,r_i))}{r_i}\leq 2^{-i}$, and in addition
\[
\frac{\mu (B(x,r_i)\cap E_\varepsilon (x))}{\mu (B(x,r_i))}<2^{-i}
\quad\text{and}\quad \frac{\mu (B(y,r_i)\cap E_\varepsilon (y))}{\mu (B(y,r_i))}<2^{-i}.
\]
For each $i$ let $\mathrm{R}_i(x):=B(x,r_i)\backslash B(x,r_i/2)$ and denote by $\Gamma_i (x)$ the collection 
of all curves $\gamma\in\Gamma_{x,y}$ such that
$$\mathcal{H}^1\big( \gamma^{-1}(\mathrm{R}_i(x)\backslash E_\varepsilon (x))\big) =0$$
and define $\Gamma_i (y)$ analogously, replacing $x$ by $y$. 
Now use the fact that $\mu$ is doubling and $\Gamma_{xy}$ 
is a Semmes family of curves to obtain
\begin{eqnarray*}
\frac{r_i}{2}\sigma_{x,y}(\Gamma _i(x))&\le&\int_{\Gamma_{xy}}\ell (\gamma\cap\mathrm{R}_i(x)\cap E_\varepsilon (x))\sigma_{xy}(\gamma )\\
&\le & \int_{CB_{x,y}\cap E_\varepsilon (x)\cap\mathrm{R}_i(x)}R_{x,y}(z)d\mu (z) \\
&\leq & \int_{CB_{x,y}\cap E_\varepsilon (x)\cap \mathrm{R}_i(x)}\left(\frac{d(x,z)}{\mu (B(x,d(x,z)))}+\frac{d(y,z)}{\mu (B(y,d(y,z)))}\right)d\mu (z)\\
&\leq &\int_{CB_{x,y}\cap E_\varepsilon (x)\cap \mathrm{R}_i(x)}\left(\frac{r_i}{\mu (B(x,r_i/2))}+\frac{2C_dd(x,y)}{\mu (B(x,d(x,y)/2))}\right)d\mu (z)\\
&\leq &\frac{r_i}{\mu (B(x,r_i/2))}\mu (E_\varepsilon (x)\cap B(x,r_i))+\frac{2C_dd(x,y)}{\mu (B(x,d(x,y)/2))}\mu (B(x,r_i))\\
&\leq& r_iC_d2^{-i}+\frac{2C_dd(x,y)}{\mu (B(x,d(x,y)/2))}\mu (B(x,r_i)). 
\end{eqnarray*}
We note that the above estimate also fills in the gap found in the proof for the real-valued case 
in~\cite[page~243]{DEKS}.
Set $C_{x,y}:=C_d+\tfrac{2C_dd(x,y)}{\mu (B(x,d(x,y)/2))}$.
Then from the above argument we see that 
\[
\sigma_{x,y}(\Gamma_i (x))\leq 2C_d\,2^{-i}+\frac{4C_dd(x,y)}{\mu (B(x,d(x,y)/2))}\frac{\mu( B(x,r_i))}{r_i}\leq 2C_{x,y}\,2^{-i}. 
\]
Then for each positive integer $n$ we have
\[
\sigma_{x,y}\left( \bigcup_{i=n}^\infty \Gamma_i(x)\right)\le C_{x,y}\, 2^{1-n}.
\]
Define
\[
\Gamma (x):=\bigcap_{n\in\N}\bigcup_{i=n}^\infty \Gamma_i(x).
\]
It follows that $\sigma_{x,y}(\Gamma (x))=0$. 
When
$\gamma\in\Gamma_{x,y}\backslash \Gamma (x)$, there exists a positive integer $n_0$ 
such that $\gamma\notin \Gamma_i(x)$ for every $i\geq n_0$. 
It suffices to have $\gamma\notin\Gamma_i(x)$ for some $i\ge n_0$ to get that there exists 
$\hat{x}\in \gamma\backslash (E_\varepsilon (x)\cup N_\gamma )$ for any $\mathcal{H}^1$-null set $N_\gamma$ in $\gamma$. 
Now consider the same argument replacing $x$ by $y$ in order to construct $\Gamma (y)$. Note that
$\sigma_{x,y}(\Gamma(x)\cup\Gamma(y))=0$.

Recall that condition \eqref{eq:AMBV} holds for every nonconstant, compact, rectifiable curve.  Therefore, for every curve 
$\gamma\in \Gamma_{xy}\backslash (\Gamma (x)\cup \Gamma (y))$ there exists an $\mathcal{H}^1$-null set $N_\gamma$ such that
\[
\Vert u(\gamma (s))-u(\gamma (t))\Vert\leq\liminf_{k\to\infty}\int_{\gamma}\rho_k\, ds
\]
whenever $s,t\in\mathrm{dom}(\gamma )\backslash\gamma^{-1}(N_\gamma )$. Since 
$\gamma\notin \Gamma (x)\cup\Gamma (y)$, there exist 
$\hat{x}\in \gamma\backslash (E_\varepsilon (x)\cup N_\gamma )$ and 
$\hat{y}\in \gamma\backslash (E_\varepsilon (y)\cup N_\gamma )$ such that
\begin{equation}\label{eq:SemmesUB}
\Vert u(x)-u(y)\Vert \leq\Vert u(\hat{x})-u(\hat{y})\Vert+2\eps\leq \liminf_{k\to\infty}\int_\gamma \rho_k ds+2\eps.
\end{equation}
(Notice that we can actually get not only such $\hat{x}$ and $\hat{y}$ but two sequences of points 
$x_i\notin E_\varepsilon (x)\cup N_\gamma$ and $y_i\notin E_\varepsilon (y)\cup N_\gamma$ 
converging to $x$ and $y$ respectively, but we do not need that here).
By the Semmes family inequality, we have
\[
\int_{\Gamma_{xy}}\int_\gamma \rho_k \, ds\, d\sigma_{x,y} (\gamma )\leq C\int_{CB_{x,y}}\rho_k(z) R_{xy}(z)d\mu (z)
\]
for each $x,y\in X\setminus N$, $k\in\N$. Therefore, by $\sigma_{x,y}(\Gamma (x)\cup \Gamma (y))=0$ and 
by~\eqref{eq:SemmesUB}, we see that
\[ 
 \Vert u(x)-u(y)\Vert=\int_{\Gamma_{x,y}}\Vert u(x)-u(y)\Vert d\sigma_{x,y} (\gamma)
\le   
C\, \liminf_{k\to\infty}\int_{CB_{x,y}}\rho_k(z) R_{xy}(z)d\mu (z)+2\eps.
\] 
Recall that $\mu(N)=0$. Now, for each ball $B\subset X$,
\begin{align}\label{eq:Semmes-Max1}
\vint{B}\Vert u-u_B\Vert d\mu&\leq\vint{B}\vint{B}\Vert u(x)-u(y)\Vert d\mu (y)d\mu (x)\notag\\
&\leq C\, \vint{B}\vint{B}\liminf_{k\to\infty}\int_{CB_{x,y}}\rho_k(z) R_{x,y}(z)d\mu (z)d\mu(y)d\mu(x)+2\eps\notag\\
&\leq \frac{C}{\mu(B)^2}\, \liminf_{k\to\infty}\int_{B}\int_{B}\int_{4CB}\rho_k(z) R_{x,y}(z)d\mu (z)d\mu(y)d\mu(x)+2\eps\notag\\
&= \frac{C}{\mu(B)^2}\,\liminf_{k\to\infty}\int_{4CB}\rho_k(z)\int_{B}\int_{B}R_{x,y}(z)d\mu(y)d\mu(x)d\mu (z)+2\eps,
\end{align}
where we have used Tonelli's theorem in the last equality.  

Now, to obtain an estimate for the inner two integrals above, we fix $z\in 4CB$.
Let $R:=\rad(B)$. By the doubling property of $\mu$, we have 
with $B_i=B(z,5C\, 2^{-i}R)$ for $i=0,1,\cdots$,
\begin{align*}
\int_B\int_B \frac{d(x,z)}{\mu(B(x,d(x,z)))}\, d\mu(y)\, d\mu(x)&=\mu(B)\, \int_B\frac{d(x,z)}{\mu(B(x,d(x,z)))}\, d\mu(x)\\
&\le \mu(B)\, \int_{B(z,5CR)} \frac{d(x,z)}{\mu(B(x,d(x,z)))}\, d\mu(x)\\
&\lesssim\, \mu(B)\, \sum_{i=0}^\infty\int_{B_i\setminus B_{i+1}} \frac{2^{-i}R}{\mu(B(z,2^{-i}R))}\, d\mu(x)\\
&\lesssim \mu(B)\, \sum_{i=0}^\infty 2^{-i}R\\
&\lesssim \mu(B)\, R,
\end{align*}
where we have implicitly used the fact that $\mu(\{w\})=0$ for each $w\in X$. The comparison constants above depend solely on
the doubling constant of $\mu$ and the constant $C$.
A similar estimate also gives
\[
\int_B\int_B \frac{d(y,z)}{\mu(B(y,d(y,z)))}\, d\mu(y)\, d\mu(x)\lesssim\mu(B)\, R.
\]
Now from~\eqref{eq:Semmes-Max1} we see that
\[
\vint{B}\Vert u-u_B\Vert\, d\mu\lesssim \frac{C\, \rad(B)}{\mu(B)}\, \liminf_{k\to\infty}\int_{4CB}\rho_k(z)\, d\mu(z)\, +\, 2\eps.
\]
Taking $\eps\to 0,$ we have that 
\[
\vint{B}\| u-u_B\|d\mu\le C\rad(B)\liminf_{k\to\infty}\vint{4CB}\rho_kd\mu.
\]

Now, let $\{\rho_k\}_{k=1}^\infty$ be an AM-upper bound for $u.$  That is, \eqref{eq:AMBV} holds for AM-a.e.\ curve.  
Then by Remark~\ref{rem:AllCurves}, we know that
there exists a sequence of nonnegative Borel functions $\{g_k\}_{k=1}^\infty$ with 
$\limsup_{k\to\infty} \int_X g_k\, d\mu<\infty$
such that for all $\eps>0,$ $\{\rho_k+\eps g_k\}_{k=1}^\infty$ is a strong bounding sequence for $u$, that is, 
where~\eqref{eq:AMBV} holds for all curves. Applying the above result, we obtain
\begin{align*}
\vint{B}\| u-u_B\|d\mu&\le C\rad(B)\liminf_{k\to\infty}\vint{4CB}(\rho_k+\eps g_k)d\mu\\
	&\le C\rad(B)\left(\liminf_{k\to\infty}\vint{4CB}\rho_kd\mu+\eps\limsup_{k\to\infty}\vint{4CB}g_kd\mu\right)\\
	&\le C\rad(B)\left(\liminf_{k\to\infty}\vint{4CB}\rho_kd\mu+\frac{\eps}{\mu(4CB)}\limsup_{k\to\infty}\int_X g_kd\mu\right).
\end{align*}
Taking $\eps\to 0^+$ yields the desired result.
\end{proof}

\begin{corollary}\label{cor:Poincare}
The following are equivalent whenever $(X,d,\mu )$ is a metric measure space with $\mu$ a doubling measure.
\begin{itemize}
\item[$(i)$] $X$ supports an AM--Poincar\'e inequality for every Banach space $V$ target.
\item[$(ii)$] $X$ supportas an AM--Poincar\'e inequality for some Banach space $V$ target.
\item[$(iii)$] $X$ supports an AM--Poincar\'e inequality for real-valued functions.
\item[$(iv)$] $X$ supports a Semmes pencil of curves.
\item[$(v)$] $X$ supports a $1$-Poincar\'e inequality.
\end{itemize}
\end{corollary}

\begin{proof}
$(i)\Rightarrow (ii)$ is immediate. If $(ii)$ holds, then in particular $X$ supports a 
$1$-Poincaré inequality for a Banach space-valued functions, and hence
supports a $1$-Poincar\'e inequality for real-valued functions, as $\R$ can be isometrically embedded
into that Banach space, that is, $(ii)\Rightarrow (v)$.
From~\cite[Theorem~3.10]{DEKS} we know that $(v)$, $(iv)$, and $(iii)$ are equivalent.
Finally, $(iv)\Rightarrow (i)$ 
follows from Theorem~\ref{thm:semmesBVAM}.
\end{proof}

The condition $\liminf_{r\to 0^+}\mu(B(x,r))/r=0$ for $\mu$-a.e.~$x\in X$ precludes us from considering spaces that
have components that are one-dimensional in nature, as for example in $\R$ and graphs. It is perhaps possible to handle 
this situation separately, as was shown for real-valued BV functions in~\cite{LZ}. We do not do so here.

\section{Proof of theorem \ref{thm:main}}\label{Sec:4}

The focus of this section is to complete the proof of the first main theorem of the paper, Theorem~\ref{thm:main}.
By Lemma~\ref{lem:BVinBVAM} we have seen that $BV(X:V)\subset BV_{AM}(X:V)$ with the energy seminorm control
$\Vert D_\mathrm{AM}u\Vert (X)\leq\Vert Du\Vert (X)$. Thus it only remains to show the reverse inequalities.
To this end, let $u\in BV_{AM}(X:V)$. We will make use of the version of Poincar\'e inequality identified in
Lemma~\ref{eq:PI-BV-AM} above.

Since the measure $\mu$ is doubling, for each $\eps>0$ there is a countable covering $\{B_i\}_i$ of $X$ by balls of radius $\eps$
such that for each $T\ge1$ there is a constant $C_T>0$, depending solely on $T$ and the doubling constant associated
with $\mu$, such that $\sum_i\chi_{TB_i}\le C_T$ on $X$. Moreover, for each $i$ there is a non-negative $C/\eps$-Lipschitz
function $\pip_i$, with support in $2B_i$, so that $\sum_i\pip_i=1$ on $X$; see for example the discussion
at the beginning of~\cite[Section~9.2]{HKSTbook}. Such a collection of functions $\{\pip_i\}_i$ is called a Lipschitz
partition of unity in $X$. Using this Lipschitz partition of unity, we now construct a locally Lipschitz continuous approximation
of $u$ as follows:
\[
u_\eps:=\sum_i u_{B_i}\, \pip_i,\qquad \text{ where }\qquad u_{B_i}:=\vint{B_i}u\, d\mu.
\]
Let $x\in X$ and fix an index $j$ such that $x\in B_j$. Then it follows that whenever $\pip_i(x)\ne 0$, necessarily $x\in 2B_i$ and
so $2B_i\cap B_j$ is non-empty; in this case, $2B_i\subset 5B_j$. Hence, using also the fact that
$u(x)=\sum_iu(x)\pip_i(x)$, we obtain
\begin{align*}
u_\eps(x)-u(x)=\sum_i[u_{B_i}-u(x)]\pip_i(x)
&=\sum_{i; 2B_i\cap B_j\ne\emptyset}\, [u_{B_i}-u(x)]\pip_i(x)\\
&=\sum_{i; 2B_i\cap B_j\ne\emptyset}\, \pip_i(x)\, \vint{B_i}[u-u(x)]\, d\mu.
\end{align*}
Thus by the doubling property of $\mu$ and the bounded overlap property of the balls $\{5B_j\}_j$, we obtain
\begin{align*}
\|u_\eps(x)-u(x)\|&\le \sum_{i; 2B_i\cap B_j\ne\emptyset}\, \vint{B_i}\|u-u(x)\|\, d\mu\\
 &\le \sum_{i; 2B_i\cap B_j\ne\emptyset}\, \vint{B_i}[\|u-u_{5B_j}\|\, d\mu+\|u_{5B_j}-u(x)\|\\
 &\lesssim \vint{5B_j}\|u-u_{5B_j}\|\, d\mu+\|u_{5B_j}-u(x)\|,
\end{align*}
and integrating over $B_j$ and summing up over $j$, and using the fact that $\{B_j\}_j$ is a cover of $X$, we obtain
\begin{align*}
\int_X\|u_\eps(x)-u(x)\|\, d\mu(x)&\le \sum_j\int_{B_j}\|u_\eps(x)-u(x)\|\, d\mu(x)\\
&\lesssim \sum_j\int_{B_j} \left(\vint{5B_j}\|u-u_{5B_j}\|\, d\mu+\|u_{5B_j}-u(x)\|\right)\, d\mu(x)\\
&\lesssim\sum_j\left(\mu(B_j)\, \vint{5B_j}\|u-u_{5B_j}\|\, d\mu+\mu(B_j)\, \vint{5B_j}\|u_{5B_j}-u(x)\|\, d\mu(x)\right)\\
&\lesssim\sum_j\mu(B_j)\, \vint{5B_j}\|u-u_{5B_j}\|\, d\mu\\
&\lesssim\sum_j \eps\, \Vert D_{AM}u\Vert(5\lambda B_j)\\
&\lesssim \eps\, \Vert D_{AM}u\Vert(X).
\end{align*}
In obtaining the penultimate inequality above, we used the AM-Poincar\'e inequality, and in obtaining the last inequality
above, we relied on the bounded overlap of the collection $\{5\lambda B_j\}_j$. Thus $u_\eps\to u$ in $L^1(X:V)$
as $\eps\to 0^+$. As $u_\eps$ is locally Lipschitz continuous (as we will show next) on the separable metric space $X$,
it follows that $u_\eps$ is Bochner measurable, and so the convergence holds in $L^1(X:V)$.

To show that $u_\eps\in N^{1,1}(X:V)$, it suffices to show that $u_\eps$ is locally Lipschitz continuous on $X$ with its
local Lipschitz constant function $\Lip u_\eps\in L^1(X)$. Here,
\[
\Lip u_\eps(x):=\limsup_{y\to x}\frac{\|u(y)-u(x)\|}{d(x,y)}.
\]
To do so, we fix $x\in X$ and choose an index $j$ such that $x\in B_j$. Considering $y\in B_j$ as well, we see that
\[
u_\eps(y)-u_\eps(x)=\sum_{i; 2B_i\cap B_j\ne\emptyset}\, u_{B_i}\, \left(\pip_i(x)-\pip_i(y)\right)
=\sum_{i; 2B_i\cap B_j\ne\emptyset}\, \left(u_{B_i}-u_{5B_j}\right)\, \left(\pip_i(x)-\pip_i(y)\right).
\]
Using the Lipschitz property of the functions $\pip_i$, we now see by the Poincar\'e inequality that
\begin{align*}
\|u_\eps(y)-u_\eps(x)\|\lesssim\frac{d(x,y)}{\eps}\, \sum_{i; 2B_i\cap B_j\ne\emptyset}\, \|u_{B_i}-u_{5B_j}\|
&\lesssim\frac{d(x,y)}{\eps}\, \sum_{i; 2B_i\cap B_j\ne\emptyset}\, \vint{B_i}\|u-u_{5B_j}\|\, d\mu\\
&\lesssim\frac{d(x,y)}{\eps}\, \vint{5B_j}\|u-u_{5B_j}\|\, d\mu\\
&\lesssim d(x,y)\, \frac{\Vert D_{AM}u\Vert(5\lambda B_j)}{\mu(B_j)}.
\end{align*}
It follows that
\[
\Lip u_\eps(x)\lesssim \inf_{j\, :x\in B_j}\, \frac{\Vert D_{AM}u\Vert(5\lambda B_j)}{\mu(B_j)}.
\]
Thus $u_\eps$ is locally Lipschtiz continuous on $X$, and it only remains to show that $\Lip u_\eps\in L^1(X)$. 
Using the fact that $\{B_j\}_j$ covers $X$, we see that
\[
\int_X\Lip u_\eps\, d\mu\lesssim\sum_j\int_{B_j} \frac{\Vert D_{AM}u\Vert(5\lambda B_j)}{\mu(B_j)}\, d\mu
 =\sum_j \Vert D_{AM}u\Vert(5\lambda B_j)\lesssim\Vert D_{AM}u\Vert(X)<\infty,
\]
and this completes the proof that $u_\eps\in L^1(X)$. As $u_\eps\to u$ in $L^1(X:V)$ and as
$\sup_\eps\, \int_X\Lip u_\eps\, d\mu\lesssim \Vert D_{AM}u\Vert(X)<\infty$, it follows that 
$u\in BV(X:V)$ with $\Vert Du\Vert(X)\lesssim \Vert D_{AM}u\Vert(X)$, completing the proof of 
Theorem~\ref{thm:main}. Note that the comparison constant in the above inequality depends solely on
the doubling constant of the measure $\mu$ and the constants from the Poincar\'e inequality.

\section{Approximate continuity and jump sets; proof of Theorem~\ref{thm:main2}}\label{sec:JumpSets}

Throughout this section, in addition to the measure $\mu$ being doubling and supporting a
$1$-Poincar\'e inequality, we will also assume that $X$ is complete.
In this section we consider the regularity properties of functions in the class $BV_{AM}(X:Y)$, 
with $(Y,d_Y)$ a proper metric space (that is,
closed and bounded subsets of $Y$ are compact). As seen from the examples in Subsection~\ref{Sec:MetricBV}, 
when $Y$ is not a Banach space, it is more natural to consider the class $BV_{AM}(X:Y)$ rather than $BV(X:Y)$.

For functions $u$ in the class $L^1(X:Y)$, by isometrically embedding $Y$ into a Banach space if necessary, we know that
for $\mu$-almost every $x\in X$, we have the Lebesgue point property hold:
\[
\limsup_{r\to 0^+}\vint{B(x,r)}d_Y(u(y),u(x))\, d\mu(y)=0,
\]
and we refer the interested reader for more on this topic to~\cite[Page~77]{HKSTbook}. At such points $x$, as in the proof of
Theorem~\ref{thm:semmesBVAM}, if we set $E_\eps(x):=\{y\in X\, :\, d_Y(u(y),u(x))>\eps\}$, then
we have that
\[
\limsup_{r\to 0^+}\frac{\mu(B(x,r)\cap E_\eps(x))}{\mu(B(x,r))}=0.
\]

\begin{definition}\label{def:AppCont}
We say that a point $x\in X$ is a point of approximate continuity of $u$ if for every $\eps>0$ we have
\[
\limsup_{r\to 0^+}\frac{\mu(\{y\in B(x,r)\, :\, d_Y(u(x),u(y))\ge \eps\})}{\mu(B(x,r))}=0.
\]
\end{definition}

The discussion from the previous paragraph tells us that $\mu$-almost every point in $X$ is a point of approximate
continuity of $u\in L^1(X:Y)$. For functions $u\in BV_{AM}(X:Y)$ we would like a better control. We may broaden
the definition of approximate continuity by saying that $u$ is approximately continuous at $x$ if there is some $y_0\in Y$
such that for every $\eps>0$ we have
\[
\limsup_{r\to 0^+}\frac{\mu(\{y\in B(x,r)\, :\, d_Y(y_0,u(y))\ge \eps\})}{\mu(B(x,r))}=0.
\]
Since $\mu$-almost every point in $X$ is a point of approximate continuity of $u$, if $x\in X$ such that there is some
$y_0$ satisfying the above density condition, then we can re-define $u$ at $x$ by setting $u(x):=y_0$; such a
modification is a better representative of $u$; moreover, such a modification needs to be done only on a set of
$\mu$-measure zero, thanks to the Lebesgue differentiation theorem mentioned above.

Note that if $x$ is a point of approximate continuity in the above sense with $y_0$ the corresponding value,
and if $Y$ is bounded, 
then for each $\eps>0$ we have 
\begin{align*}
\limsup_{r\to 0^+}\vint{B(x,r)}d_Y(u(z),y_0)\, d\mu(z)
\le \eps+\diam(Y)\, \limsup_{r\to 0^+}\frac{\mu(\{y\in B(x,r)\, :\, d_Y(y_0,u(y))\ge \eps\})}{\mu(B(x,r))}
=\eps,
\end{align*}
and so necessarily $x$ is a Lebesgue point of $u$ as well.

The notions of approximate continuity and jump values as considered 
in~\cite[page~294]{Ambrosio1990} is somewhat different than ours in that there it is required that for every
continuous function $g:Y\to\R$ the map $g\circ u$ is approximately continuous at $x$ with the approximate
limit being $g(y_0)$. Such an indirect definition seems to be not needed here, and we take the definition 
of approximate continuity proposed by Ambrosio in~\cite[Definition~1.1]{Ambrosio1990-II}.

Let us consider points $x\in X$ that are not points of approximate continuity in the above, more expanded, sense. 
We would like to call such points $x$ jump points of $u$. So the jump set $\mathcal{J}(u)$ 
of $u$ is the collection of all points in $X$
that are not points of approximate continuity, in the above broader sense, of $u$. We would like to know that 
for $x\in \mathcal{J}(u)$ there are finitely many points 
$y_1,y_2,\cdots, y_k$, with $k\le k_0$ where $k_0$ is independent of $u$ and $x$,
that act as jump values of $u$ at $x$. This may not be possible at all $x\in \mathcal{J}(u)$, but we would like to ensure that
this is possible for $\mathcal{H}^{-1}$-a.e.~$x\in \mathcal{J}(u)$.
An additional drawback in~\cite{Ambrosio1990}
is that the discussion regarding jump sets is incomplete; if we know that there is a set $F$, of positive density 
at $x\in \mathcal{J}(u)$,
for which $g\circ u$ takes on an approximate limit at $x$ along $F$ for \emph{each} continuous $g:Y\to\R$, then 
from~\cite[Proposition~1.1]{Ambrosio1990} we know that $u$ has a jump value at $x$. The existence of such $F$ is not
considered there. For these reasons we do not rely in the techniques and  
results of~\cite{Ambrosio1990} in our discussion, but work directly with the maps $u$ themselves.
The discussion in~\cite{Ambrosio1990-II} is set in the context of $X=\R^n$ and relies on the structure of sets of finite
perimeter there; unlike in the Euclidean setting, even real-valued BV function in more general metric spaces $X$ 
considered here may have more than two jump values. This construction of the set $\mathcal{J}(u)$ immediately
implies the validity of Claim~(a) of Theorem~\ref{thm:main2}.

We first make the simplifying reduction that $Y$ is a compact metric space. 
We refer the interested reader to the final section of the paper, the appendix, for the final step that 
allows us to extend the result to non-compact proper metric space $Y$.
Now, if $x\in \mathcal{J}(u)$, 
then for every $y\in Y$
there is some $\eps_y>0$ such that 
\[
\limsup_{r\to 0^+}\frac{\mu(\{z\in B(x,r)\, :\, d_Y(y, u(z))\ge \eps_y\})}{\mu(B(x,r))}>0.
\]
For each $y\in Y$ and $\eps>0$ set 
\begin{equation}\label{eq:Eyeps}
F(y,\eps):=\{z\in X\, :\, d_Y(u(z),y)\ge \eps\}.
\end{equation}
Note that then for every $0<\eps\le \eps_y$ we have that 
\[
\limsup_{r\to 0^+}\frac{\mu(B(x,r)\cap F(y,\eps))}{\mu(B(x,r))}>0.
\]
We fix $\eps>0$, and cover the compact set $Y$ by finitely many balls $B(y_i,\eps)$, $i=1,\cdots, N_\eps$. Note that as
$B(x,r)=\bigcup_{i=1}^{N_\eps}u^{-1}(B(y_i,\eps)$, necessarily there is some $y_1\in Y$, relabeled if necessary, such that
\[
\limsup_{r\to 0^+}\frac{\mu(B(x,r)\cap u^{-1}(B(y_1,\eps)))}{\mu(B(x,r))}>0.
\]
If we also have that 
\begin{equation}\label{eq:dense-compl-1}
\limsup_{r\to 0^+}\frac{\mu(B(x,r)\cap u^{-1}(Y\setminus B(y_1,3\eps)))}{\mu(B(x,r))}>0,
\end{equation}
then we have two sets $E_1,E_2\subset X$ with $E_1=u^{-1}(B(y_1,\eps))$ and $E_2=u^{-1}(B(w_1,\eps))$,
$d_Y(y_1,w_1)\ge 3\eps$, such that 
\begin{equation}\label{eq:dense-at-x}
\limsup_{r\to 0^+}\frac{\mu(B(x,r)\cap E_1)}{\mu(B(x,r))}>0,\text{ and }
\limsup_{r\to 0^+}\frac{\mu(B(x,r)\cap E_2)}{\mu(B(x,r))}>0.
\end{equation}
If~\eqref{eq:dense-compl-1} fails, then we know that 
\[
\lim_{r\to 0^+}\frac{\mu(B(x,r)\cap u^{-1}(Y\setminus B(y_1,3\eps)))}{\mu(B(x,r))}=0
\]
and so 
\[
\lim_{r\to 0^+}\frac{\mu(B(x,r)\cap u^{-1}(\overline{B}(y_1,3\eps)))}{\mu(B(x,r))}=1;
\]
In this case, we can cover the compact set $\overline{B}(y_1,3\eps)$ by balls of radii $\eps/6^2$, and obtain
a point $y_2\in \overline{B}(y_1,3\eps)$ so that
\[
\limsup_{r\to 0^+}\frac{\mu(B(x,r)\cap u^{-1}(B(y_2,\eps/6^2)))}{\mu(B(x,r))}>0.
\]
If we know that 
\[
\limsup_{r\to 0^+}\frac{\mu(B(x,r)\setminus u^{-1}(B(y_2,3\eps/6^2)))}{\mu(B(x,r))}>0,
\]
then we can set $E_1=u^{-1}(B(y_2,\eps/6^2))$ and $E_2=u^{-1}(B(w_2,\eps/6^2))$, with
$6\eps\ge d_Y(y_2,w_2)\ge 3\eps/6^2$ 
such that~\eqref{eq:dense-at-x} holds. If the above analog of~\eqref{eq:dense-compl-1}
fails, then we know that
\[
\lim_{r\to 0^+}\frac{\mu(B(x,r)\cap u^{-1}(\overline{B}(y_2,3\eps/6^2)))}{\mu(B(x,r))}=1,
\]
and the process inductively continues. Thus we obtain a sequence of points $y_1, y_2,\cdots$ with 
$d_Y(y_i,y_{i+1})\le 3\eps/6^i$ and so that
\[
\lim_{r\to 0^+}\frac{\mu(B(x,r)\cap u^{-1}(\overline{B}(y_i,3\eps/6^i)))}{\mu(B(x,r))}=1.
\]
If this process continues ad infinitum, then we obtain a Cauchy sequence $\{y_i\}_i$ in $Y$ which,
by the completeness of $Y$, must converge to a point $y_\infty$ for which we would have that
for each $\tau>0$,
\[
\limsup_{r\to 0^+}\frac{\mu(B(x,r)\cap u^{-1}(B(y_\infty,\tau)))}{\mu(B(x,r))}=1,
\]
and so re-setting $u(x)=y_\infty$ would show that $x\not\in \mathcal{J}(u)$. Therefore the inductive process
above must terminate at some index $k$, and so we know that there is some $y_k, w_k\in Y$ 
such that $d_Y(y_k,w_k)\ge 3\eps/6^k$, and with $E_1=u^{-1}(B(y_k,\eps/6^k))$ and 
$E_2=u^{-1}(B(w_k,\eps/6^k))$, condition~\eqref{eq:dense-at-x} holds. Note that
$\dist(B(y_k,\eps/6^k), B(w_k,\eps/6^k))\ge \eps/6^k>0$.
Hence, the following is an equivalent definition of a jump point of $u\in BV_{AM}(X:Y)$.

\begin{definition}\label{def:jump}
Let $u:X\rightarrow Y$. We say that $x_0\in X$ is a jump point of $u$ if there exist sets $E_1,E_2\subset X$ such that
\begin{equation}\label{eq:jumpdense}
\limsup_{r\to 0}\frac{\mu (B(x_0,r)\cap E_i)}{\mu (B(x_0,r)}>0\ \text{ for }\ i=1,2,
\end{equation}
and there exist balls $B_1,B_2\subset Y$ with $\mathrm{dist}(B_1,B_2)\ge \rad(B_1)$ such that 
$u(E_i)\subset B_i$ for $i=1,2$.
\end{definition}

From the discussion preceding the above definition, we know that $x\in\mathcal{J}(u)$ if and only if
$x$ is a jump point in the sense of Definition~\ref{def:jump} above.

From Lemma~\ref{lem:Leibnitz} we obtain the following.

\begin{lemma}
Let $u\in BV_{AM}(X:Y)$, and $y_0\in Y$. Then $v:X\to\R$ given by $v(x)=d_Y(u(x),y_0)$ belongs to the class 
$BV_{AM}(X)=BV(X)$.
\end{lemma}

As a corollary to the above lemma, the co-area formula from Lemma~\ref{lem:coarea} yields the 
following, from which we obtain $\sigma$-finiteness of the jump set with respect to $\mathcal{H}^{-1}$.

\begin{corollary}\label{cor:levelsets}
Let $u\in BV_{AM}(X:Y)$.  For each $y\in Y$ and $\rho>0$ set $E(y,\rho):=u^{-1}(B(y,\rho))$. 
Then for each $y\in Y$ there is a set $D_y\subset[0,\infty)$
with $\mathcal{L}^1(D_y)=0$ such that for each $\rho\in (0,\infty)\setminus D_y$ we have that $E(y,\rho)$ is of finite perimeter
in $X$.
\end{corollary}

The next result proves that the set $\mathcal{J}(u)$, as constructed above, satisfies its $\sigma$-finiteness
with respect to the co-dimensional measure $\mathcal{H}^{-1}$ claimed in the statement of Theorem~\ref{thm:main2}.

\begin{corollary}\label{cor:SigmaFinite}
For each $u\in BV_{AM}(X:Y)$, the jump set $\mathcal{J}(u)$ is 
$\sigma$-finite with respect to the co-dimension $1$ Hausdorff measure $\mathcal{H}^{-1}$ on $X$.
\end{corollary}

\begin{proof}
As $Y$ is separable, there exists a countable dense subset $Y_0$ of $Y$, and for each $y\in Y_0$, let 
\[
\mathcal{R}(y):=\{\rho>0: P(E(y,\rho),X)<\infty\}.
\]
By Corollary~\ref{cor:levelsets}, we have that $\mathcal{L}((0,\infty)\setminus\mathcal{R}(y))=0$, 
and so there exists a countable subset $\mathcal R_0(y)\subset\mathcal{R}(y)$ dense in $(0,\infty)$.  By Lemma~\ref{lem:reducedBdy}, it follows that $\mathcal{H}^{-1}(\partial_*(E(y,\rho))<\infty$ for each 
$\rho\in\mathcal{R}(y)$, where $\partial_*E(y,\rho)$ is the measure-theoretic boundary of $E(y,\rho)$, 
as given by~\eqref{eq:MeasureTheoreticBoundary}.

Now, for each $x\in\mathcal{J}(u)$, we have by Definition~\ref{def:jump} and the density of $Y_0$ in $Y$, that there exists $y_1,y_2\in Y_0$, $\rho_1\in\mathcal{R}_0(y_1)$, and $\rho_2\in\mathcal{R}_0(y_2)$ such that $u(E_1)\subset B_1\subset B(y_1,\rho_1)$, $u(E_2)\subset B_2\subset B(y_2,\rho_2)$, and $\dist(B(y_1,\rho_1),B(y_2,\rho_2))>0$.  Here $E_1$, $E_2$, $B_1$, and $B_2$ are as given in Definition~\ref{def:jump}.  Then, we have that $x\in\partial_* E(y_1,\rho_1)$, and so it follows that 
\[
\mathcal{J}(u)\subset\bigcup_{y\in Y_0}\bigcup_{\rho\in\mathcal{R}_0(y)}\partial_* E(y,\rho).\qedhere
\]   
\end{proof}

The above notion of jump sets agrees with the notion of jump sets for real-valued BV functions, see for 
example~\cite{Ambrosio, LSh1, EGLS} for real-valued BV functions in the metric setting, and~\cite{EvansGariepy}
for the Euclidean setting. The discussion towards the end of this section gives a brief overview of why these
notions agree.
However, as pointed out in~\cite{LSh1}, a BV function can take on infinitely many values
near the jump point, but such a bad behavior cannot happen on a large set. To demonstrate a similar behavior of
metric space-valued BV functions, we first consider what it means for a point in the target metric space to be a jump value
near a jump point of the BV function.

\begin{definition}
With $u\in BV_{AM}(X:Y)$ and $x\in\mathcal{J}(u)$, we say that a point $y_0\in Y$ is a jump value of $u$ at $x$ if
for every $\eps>0$ we have that
\[
\limsup_{r\to 0^+}\frac{\mu(B(x,r)\cap u^{-1}(B(y_0,\eps)))}{\mu(B(x,r))}>0.
\]
\end{definition}

The next proposition verifies the claim~(b) of Theorem~\ref{thm:main2}.

\begin{prop}\label{prop:finitejumps} 
There exists $k_0\in\N$ so that for every $u\in BV_{AM}(X:Y)$ there is a set $N\subset X$ with 
$\mathcal{H}^{-1}(N)=0$
such that for each $x\in\mathcal{J}(u)\setminus N$ there are at least two and at most 
$k_0$ jump values $y_1,\cdots,y_k\in Y$ of $u$ at $x$.  Furthermore, for every $\eps>0$ and $i=1,2,\cdots, k$, we have
\[
\liminf_{r\to 0^+}\frac{\mu(B(x,r)\cap u^{-1}(B(y_i,\eps)))}{\mu(B(x,r))}\ge \gamma,
\]
and
\[
\limsup_{r\to 0^+}\frac{\mu(B(x,r)\setminus \bigcup_{i=1}^ku^{-1}(B(y_i,\eps)))}{\mu(B(x,r))}=0.
\]
Here $k_0$ and $\gamma$ are constants depending only the doubling constant and Poincar\'e constants of $X$, 
and in particular are independent of $Y$, $u$, and $\eps$. 
\end{prop}

\begin{proof}
Since $Y$ is compact, it is separable.  As above, let $Y_0$ be a countable dense subset of $Y$, and for each $y\in Y_0$ let
\[
\mathcal{R}(y)=\{ \rho>0 :P(E(y,\rho),X)<\infty \}.
\]
Note from Corollary~\ref{cor:levelsets} that $\mathcal{L}^1((0,\infty)\setminus\mathcal{R}(y))=0$. Let $\mathcal{R}_0(y)$ be a
countable dense subset of $\mathcal{R}(y)$. For each $y\in Y_0$ and $\rho\in\mathcal{R}(y)$ we know that
$\mathcal{H}^{-1}(\partial_*E(y,\rho)\setminus \Sigma_\gamma(E(y,\rho)))=0$, where
$\partial_*E(y,\rho)$ is the measure-theoretic boundary of $E(y,\rho)$, as given by \eqref{eq:MeasureTheoreticBoundary}, and $\Sigma_\gamma(E(y,\rho))$ is the reduced boundary of $E(y,\rho)$, as given by \eqref{eq:ReducedBoundary}. 
Here $0<\gamma\le\tfrac12$ is a number that depends solely on the constants associated with the doubling property of $\mu$
and the Poincar\'e inequality; see for example~\cite[Theorem~5.3]{Ambrosio}. Let
\[
N:=\bigcup_{y\in Y_0}\, \bigcup_{\rho\in \mathcal{R}_0(y)}\, \partial_*(E(y,\rho))\setminus\Sigma_\gamma(E(y,\rho)).
\]
Then, by the countability of the collections, we have that $\mathcal{H}^{-1}(N)=0$. We now fix 
$x\in\mathcal{J}(u)\setminus N$. We proceed in an inductive fashion hinted at in the discussion preceeding 
Definition~\ref{def:jump}. 

Let $E_1$ be one of the two sets identified in Definition~\ref{def:jump}, associated with the jump point $x$,
and let $B_1$ be the corresponding ball in $Y$ such that $u(E_1)\subset B_1$ and 
\[
\limsup_{r\to 0^+}\frac{\mu(B(x,r)\cap E_1)}{\mu(B(x,r))}>0 \text{ and }
\limsup_{r\to 0^+}\frac{\mu(B(x,r)\setminus E_1)}{\mu(B(x,r))}>0.
\]
Since the distance between the balls $B_1$ and $B_2$ in Definition~\ref{def:jump} is
positive, we are free to choose the center $y_1$ of $B_1$ to be in $Y_0$ and then, by increasing the radius slightly if necessary,
have the radius of $B_1$ be in the set $\mathcal{R}_0(y_1)$. A similar modification can be made to the ball $B_2$.
We can now replace $E_1$ with $u^{-1}(B_1)$ and $E_2$ with $u^{-1}(B_2)$; hence from now on,
$E_1=u^{-1}(B_1)$.
Thus
we have that $x\in\partial_*E_1$ because of the existence of $B_2$, 
and as $x\not\in N$, we see that
\[
\liminf_{r\to 0^+}\frac{\mu(B(x,r)\cap E_1)}{\mu(B(x,r))}\ge \gamma \text{ and }
\liminf_{r\to 0^+}\frac{\mu(B(x,r)\setminus E_1)}{\mu(B(x,r))}\ge \gamma.
\]
Let $\rho>0$ be the radius of the ball $B_1$, and note that the distance between $B_1$ and $B_2$ is at least $\rho$.
Covering the closed ball $\overline{B_1}$ by balls $B(y_{2,1},\rho/12),\cdots, B(y_{2,N_2},\rho/12)$,
with $y_{2,i}\in Y_0$ for $i=1,\cdots, N_2$, and $B(y_{2,i},\rho/12)$ intersects $\overline{B_1}$.
By doing so, we can find a point $y_2\in \tfrac{13}{12}B_1$ such that
\[
\limsup_{r\to 0^+}\frac{\mu(B(x,r)\cap u^{-1}(B(y_2,\rho/12))))}{\mu(B(x,r))}>0.
\]
We can then find $\rho_2\in\mathcal{R}_0(y_2)$ such that $\rho/12\le \rho_2<\rho/11$,
so that with $E_{2,1}=u^{-1}(B(y_2,\rho_2))$, we have by the fact that $x\not\in N$,
\[
\liminf_{r\to 0^+}\frac{\mu(B(x,r)\cap E_{2,1})}{\mu(B(x,r))}\ge \gamma.
\]
In the above, we have used the fact that $B_2$ does not intersect $\overline{B_1}\cup\overline{B}(y_2,\rho_2)$ 
to know that $x\in\partial_*E_{2,1}$. We proceed inductively as follows. Once $y_i\in Y_0$ and 
$\rho_i\in\mathcal{R}_0(y_i)$, $i=1,\cdots, k$, such that $d_Y(y_i, y_{i+1})<2\rho_i$ and
$\rho_{i+1}<\rho_i/11$, and with $E_{i,1}=u^{-1}(B(y_i,\rho_i))$ we have
\[
\liminf_{r\to 0^+}\frac{\mu(B(x,r)\cap E_{i,1})}{\mu(B(x,r))}\ge \gamma
\]
for $i=2,\cdots, k$, we cover $\overline{B}(y_k,\rho_k)$ by balls 
$B(y_{k+1,1},\rho_k/12),\cdots, B(y_{k+1,N_{k+1}},\rho_k/12)$, each intersecting $\overline{B}(y_k,\rho_k)$
with $y_{k+1,i}\in Y_0$ for $i=1,\cdots, N_{k+1}$, and hence find $y_{k+1}\in Y_0$ so that 
$d(y_k,y_{k+1})<2\rho_k$, and 
\[
\limsup_{r\to 0^+}\frac{\mu(B(x,r)\cap u^{-1}(B(y_{k+1},\rho_k/12))))}{\mu(B(x,r))}>0.
\]
We then find $\rho_{k+1}\in\mathcal{R}_0(y_{k+1})$ such that $\rho_k/12\le \rho_{k+1}<\rho_k/11$, 
and hence have that with $E_{k+1,1}=u^{-1}(B(y_{k+1},\rho_{k+1}))$,
\[
\liminf_{r\to 0^+}\frac{\mu(B(x,r)\cap E_{k+1,1})}{\mu(B(x,r))}\ge \gamma.
\]
Note that as for each $j$ we have $\rho_j<(11)^{-j}\, \rho$, and as $\dist(B_1,B_2)>\rho$, necessarily
$B(y_k,\rho_k)\cap B_2=\varnothing$. Moreover, as $d_Y(y_k,y_{k+1})<(11)^{k-1}\rho$, we also have that
the sequence $\{y_j\}_j$ is a Cauchy sequence in $Y$, and as $Y$ is complete, converges to some $y_\infty\in Y$.
We now show that $y_\infty$ is a jump value of $u$ at $x$. Let $\eps>0$; then there is some positive integer $k$
so that $B(y_k,\rho_k)\subset B(y_\infty,\eps)$. It follows that
\[
\liminf_{r\to 0}\frac{\mu(B(x,r)\cap u^{-1}(B(y_\infty,\eps)))}{\mu(B(x,r))}
\ge \liminf_{r\to 0^+}\frac{\mu(B(x,r)\cap E_{k,1})}{\mu(B(x,r))}\ge \gamma>0.
\]
Thus $u$ has at least one jump value at $x$, and moreover, we also have that for each $\eps>0$,
\[
\liminf_{r\to 0}\frac{\mu(B(x,r)\cap u^{-1}(B(y_\infty,\eps)))}{\mu(B(x,r))}
\ge\gamma.
\]
Note also, from switching the roles of the sets $E_1$ and $E_2$, 
we obtain a second jump value of $u$ at $x$.

Now, if $z\in Y$ is any other jump value of $u$ at $x$, then for each $\eps>0$ with $\eps<d_Y(z,y_\infty)/20$,
we can find $z_1\in B(z,\eps/2)\cap Y_0$ and $0<\tau<\eps/4$ such that $\tau\in\mathcal{R}_0(z_1)$
and note that $u^{-1}(B(z_1,\tau))\subset u^{-1}(B(z,\eps))$ 
with
\begin{equation}\label{eq:lower-density-jump}
\liminf_{r\to 0}\frac{\mu(B(x,r)\cap u^{-1}(B(z,\eps)))}{\mu(B(x,r))}
\ge \liminf_{r\to 0}\frac{\mu(B(x,r)\cap u^{-1}(B(z_1,\tau)))}{\mu(B(x,r))}
\ge\gamma;
\end{equation}
that is, \emph{each} jump value of $u$ at $x$ satisfies the above lower density at least $\gamma$ at $x_0$.
As $\gamma>0$, there are at most $k_0:=\lceil 1/\gamma\rceil$ number of such jump values for $x$.

Now suppose that we have identified all the jump values $y_1,\cdots, y_k$ of $u$ at $x$, with $2\le k\le k_0$.
We claim that for each $\tau>0$, the set $E(\tau):=\bigcup_{j=1}^kE(y_j,\tau_i)$ has density $1$ at $x$, that is,
\[
\liminf_{r\to 0^+}\frac{\mu(B(x,r)\cap E(\tau))}{\mu(B(x,r))}=1.
\]
Here $\tau_i\in\mathcal{R}_0(y_i)$ such that $\tfrac23\tau<\tau_i\le \tau$.
It suffices to know this for all sufficiently small $\tau>0$, and so we consider $\tau>0$ for which
the closed balls $\overline{B}(y_i,\tau_i)$ are pairwise disjoint. 
If the claim does not hold, then we would have that
\[
\limsup_{r\to 0^+}\frac{\mu(B(x,r)\setminus E(\tau))}{\mu(B(x,r))}>0.
\]
Then let $0<\eps<\tau/20$ such that for each $i,j$ with $i\ne j$ we have that 
$\dist(\overline{B}(y_i,\tau_i), \overline{B}(y_j,\tau_j))>20\eps$.
Now setting $K(\tau)=X\setminus E(\tau)$, we have from~\eqref{eq:lower-density-jump} that 
\[
\liminf_{r\to 0^+}\frac{\mu(B(x,r)\setminus K(\tau))}{\mu(B(x,r))}\ge \gamma>0\, \text{ and simultaneously, }\, 
\limsup_{r\to 0^+}\frac{\mu(B(x,r)\cap K(\tau))}{\mu(B(x,r))}>0.
\]
Now, by a repeat of the covering argument employed in the first part of this proof,
we cover $Y\setminus \bigcup_{j=1}^k B(y_j,\tau_i)$ 
by finitely many balls of radii $\eps$, and so find a ball $B_1$, centered at 
$w_1\in Y\setminus \bigcup_{j=1}^k B(y_j,\tau_i)$, such that
\[
\limsup_{r\to 0^+}\frac{\mu(B(x,r)\cap K(\tau)\cap E(w_1,\eps))}{\mu(B(x,r))}>0.
\]
Then by modifying $w_1$ if necessary, we can ensure that $w_1\in Y_0$, and then find $\rho_1\in\mathcal{R}_0(w_1)$
so that $\eps\le \rho_1<\tfrac{13}{12}\eps$.  Note that 
$B(w_1,\rho_1)$ is necessarily disjoint from $u(E(\tau/2))$ by this choice.  Therefore we must have $E(w_1,\rho_1)\subset K(\tau/2)$, and so
\[
\liminf_{r\to 0^+}\frac{\mu(B(x,r)\cap K(\tau/2)\cap E(w_1,\rho_1))}{\mu(B(x,r))}
=\liminf_{r\to 0^+}\frac{\mu(B(x,r)\cap E(w_1,\rho_1))}{\mu(B(x,r))}\ge \gamma.
\]
At this point, we can repeat the preceding argument that established the existence of the jump values
to conclude that there must be a jump value in $Y$ attained by $u$ along $K(\tau)$, violating the maximality of the 
collection of jump values considered above. It follows that $K(\tau)$ has density $0$ at $x$.
\end{proof}

As pointed out in the early sections of this paper, the theory of real-valued functions of bounded variation on complete doubling
metric measure spaces supporting a $1$-Poincar\'e inequality is reasonably well-established. The notion of real-valued
functions of bounded variation in metric measure spaces was first proposed by
Miranda Jr.~in~\cite{Miranda}, and its fine properties were studied in~\cite{Ambrosio, AMP, Lah-Fed}; an elegant
account of real-valued functions of bounded variation and their fine properties in the Euclidean setting can be 
found in~\cite{EvansGariepy}.
The fine properties of real-valued BV functions studied there includes approximate continuity and jump points.
The notion of approximate continuity, as proposed in Section~\ref{sec:JumpSets}, is the same as that found in
real analysis texts and in~\cite{Ambrosio, AMP}. In this section we will verify that the notion of jump sets and jump values,
as given in Section~\ref{sec:JumpSets}, agrees with the corresponding notion as given in~\cite{Ambrosio}.

As considered in~\cite{Ambrosio},
given $u:X\rightarrow \R$ $x_0\in X$ is a jump point of $u$ if $u^\wedge (x_0)<u^\vee (x_0)$, where
\begin{eqnarray*}
u^\wedge (x_0)=\sup\left\{ t\in\R : \lim_{r\to 0}\frac{\mu (B(x_0,r)\cap \{ u\leq t\})}{\mu (B(x_0,r))}=0 \right\}, \\
u^\vee (x_0)=\inf\left\{ t\in\R : \lim_{r\to 0}\frac{\mu (B(x_0,r)\cap \{ u\geq t\})}{\mu (B(x_0,r))}=0 \right\}.
\end{eqnarray*}

\begin{lemma}
Let $u:X\rightarrow \R$. Then $x_0\in\mathcal{J}(u)$ if and only if $u^\wedge (x_0)<u^\vee (x_0)$.
\end{lemma}

\begin{proof}
Suppose first that $u^\wedge(x_0)=u^\vee(x_0)=:\beta$. It then follows that for each $\eps>0$,
\[
\lim_{r\to 0^+}\frac{\mu(B(x_0,r)\cap\{|u-\beta|\ge\eps\})}{\mu(B(x_0,r))}=0,
\]
and so by Definition~\ref{def:AppCont} the point $x_0$ is a point of approximate continuity of $u$, that is, $x_0\not\in\mathcal{J}(u)$.

For the converse suppose that $-\infty<u^\wedge (x_0)<u^\vee (x_0)<\infty$, 
and choose $t_1^-, t_1^+, t_2^-$, and $t_2^+$ such that 
$t_1^-<u^\wedge (x_0) <t_1^+<t_2^-<u^\vee (x_0)<t_2^+$. Then 
\[ 
\lim_{r\to 0}\frac{\mu (B(x_0,r)\cap \{ u\leq t_1^-\})}{\mu (B(x_0,r))}=0 \ \text{ and }\
\lim_{r\to 0}\frac{\mu (B(x_0,r)\cap \{ u\geq t_2^+\})}{\mu (B(x_0,r))}=0.
\] 
Since $u^\wedge (x_0)<t_1^+$ and $u^\vee (x_0)>t_2^-$ we also have
\[ 
\limsup_{r\to 0}\frac{\mu (B(x_0,r)\cap \{ u\leq t_1^+\})}{\mu (B(x_0,r))}>0 \ \text{ and }\
\limsup_{r\to 0}\frac{\mu (B(x_0,r)\cap \{ u\geq t_2^-\})}{\mu (B(x_0,r))}>0,
\] 
and so if we set $E_i = \{ t_i^-\leq u\leq t_i^+\}$ then
\[
\lim_{r\to 0}\frac{\mu (B(x_0,r)\cap E_i)}{\mu (B(x_0,r))}>0\ \text{ for }\ i=1,2,
\]
with $u(E_i)\subset B_i:=(t_i^-,t_i^+)$. Since $t_1^+<t_2^-$, $\mathrm{dist} (B_1,B_2)>0$. Thus $x_0$ is not
a point of approximate continuity of $u$, that is, $x_0\in\mathcal{J}(u)$. If $u^\wedge(x_0)=-\infty$ or
if $u^\vee(x_0)=\infty$, then we replace $u$ with $\chi_{K_n}\cdot u$ and proceed as above, with 
$K_n=\{|u|\le n\}$.
\end{proof}

\section{Appendix}

\subsection{The outer measure property of $\Vert D_{AM}u\Vert$}\label{section:OuterMeasure}

In \cite[Theorem~3.4]{Miranda}, Miranda Jr.\ proves the outer measure property of $\|Du\|$ for a function in $u\in BV(X)$ 
using the criterion given in DeGiorgi--Letta~\cite[Theorem~5.1]{DL}.  To do so, he relies upon a delicate 
construction by which approximating sequences of locally Lipschitz functions defined on two open sets are stitched 
together to obtain an approximating sequence defined on the union of the open sets.  By the nature of 
Definition~\ref{def:Miranda}, this must be done in a manner so that both the $L^1$-convergence and 
energies of the new sequence of functions are controlled.  In \cite[Theorem~4.1]{Martio}, Martio 
proves the outer measure property of $\|D_{AM}u\|$ for $u\in BV_{AM}(X)$ using~\cite[Theorem~5.1]{DL} in 
a similar manner.  However, the stitching argument employed there is much simpler due to definition of 
$BV_{AM}(X)$, since one only needs to stitch together the AM-upper bounds.  We include a detailed 
proof of this stitching lemma for the convenience of the reader.       

\begin{lemma}\cite[Lemma~2.4]{Martio}\label{lem:MartioLemma}
Let $\Omega_1,\Omega_2\subset X$ be open sets, and let $\{g^1_i\}_i$ and $\{g^2_i\}_i$ be AM-upper bounds for a function $u\in L^1(\Omega_1\cup \Omega_2: V)$ in $\Omega_1$ and $\Omega_2$ respectively.  Then 
\[
g_i(x)=
\begin{cases}
g^1_i(x),& x\in \Omega_1\setminus \Omega_2\\
\max\{g^1_i(x),g^2_i(x)\},& x\in \Omega_1\cap \Omega_2\\
g^2_i(x),&x\in \Omega_2\setminus \Omega_1
\end{cases}
\]
is an AM-upper bound for $u$ in $\Omega_1\cup\Omega_2.$ 
\end{lemma}

\begin{proof}
For $j=1,2,$ let $\Gamma_j$ denote the collection of curves in $\Omega_j$ for which \eqref{eq:AMBV} fails for the AM-upper bound $\{g^j_i\}_i$.  Let $\Gamma$ denote the collection of curves in $\Omega_1\cup\Omega_2$ which have a subcurve in $\Gamma_1\cup\Gamma_2.$  Then it follows that $AM(\Gamma)\le AM(\Gamma_1\cup\Gamma_2)=0.$

Let $\gamma$ be a curve in $\Omega_1\cup\Omega_2$ such that $\gamma\not\in\Gamma.$  By compactness of $\gamma([0,l(\gamma)]),$ there exists $\delta>0$ such that $\gamma'$ lies in $\Omega_1$ or $\Omega_2$ whenver $\gamma'$ is a subcurve of $\gamma$ with $l(\gamma')<\delta.$  Choose a partition $0=t_0<t_1<\cdots<t_n=l(\gamma)$ such that $t_k-t_{k-1}<\delta/2$ for $1\le k\le n.$  Since $\gamma\not\in\Gamma,$ it follows that $\gamma|_{[t_{k},t_{k+2}]}\not\in\Gamma_1\cup\Gamma_2$ for $0\le k\le n-2$. Therefore, for each such $k$, there exists a subset $N_k\subset[t_{k},t_{k+2}]$ with $\mathcal{H}^1(N_k)=0$ and a such that for all $s,t\in[t_{k},t_{k+2}]\setminus N_k,$ we have
\begin{equation}\label{eq:subcurve}
\|u(\gamma(s))-u(\gamma(t))\|\le\liminf_{k\to\infty}\int_{\gamma|_{[s,t]}}g_ids.
\end{equation}
Let $N=\bigcup_k N_k,$ and let $s,t\in [0,l(\gamma)]\setminus N,$ with $s<t.$  Then there exists $0\le k_1\le k_2\le n-1$ such that $s\in[t_{k_1},t_{k_1+1}]$ and $t\in[t_{k_2},t_{k_2+1}]$.  Let $s:=s_{k_1}$, $t:=s_{k_2}$, and for each $k_1< k<k+2,$ choose $s_k\in [t_k,t_{k+1}]\setminus N$.  By the triangle inequality and \eqref{eq:subcurve}, it follows that 
\begin{align*}
\|u(\gamma(s))-u(\gamma(t))\|\le\sum_{k=k_1}^{k_2}\|u(\gamma(s_{k}))-u(\gamma(s_{k+1}))\|&\le\sum_{k=k_1}^{k_2}\liminf_{i\to\infty}\int_{\gamma|_{[s_k,s_{k+1}]}}g_ids\\
	&\le\liminf_{i\to\infty}\int_{\gamma|_{[s,t]}}g_ids.\qedhere
\end{align*}
\end{proof}

Using Lemma~\ref{lem:MartioLemma}, Martio obtains the following using an argument analagous to the proof of \cite[Theorem~3.4]{Miranda}:

\begin{theorem}\cite[Theorem~4.1]{Martio}
If $u\in BV_{AM}(X:V),$ then $\|D_{AM}u\|(\cdot)$ (defined on open sets) defines a Borel outer measure in $X.$
\end{theorem}

\subsection{Dealing with a non-compact proper $Y$}

We now consider the case that the metric space $(Y,d_Y)$ is a proper metric space that is not compact. Recall that the proofs
and discussions in Section~\ref{sec:JumpSets} dealt with the case that $Y$ is compact as then we can focus on covering
$Y$ by \emph{finitely many} balls of radius $\eps>0$ and hence find a ball whose pre-image has positive density at
a point $x\in \mathcal{J}(u)$. If we had instead a countably infinite many balls needed to cover $Y$, then we do not
know that there must be one ball whose pre-image has positive density at $x$. When $Y$ is not compact, this is because
$Y$ is not bounded; hence we cannot cover $Y$ by finitely many balls of fixed radius $\eps>0$. In this subsection we 
point out how to deal with this situation.

As in the proof of Proposition~\ref{prop:finitejumps}, let $Y_0$ be a countable dense subset of $Y$, and for each
$y\in Y_0$ let $\mathcal{R}_0(y)$ be a countable dense subset of $\mathcal{R}(y)$. 
If there is some $R>0$ and $a\in Y$ such that $\mu(u^{-1}(Y\setminus B(a,R)))=0$, then we can replace $Y$ with
$\overline{B}(a,R)$ and the proof of Proposition~\ref{prop:finitejumps} identifies the jump values of $u$ at
points in $\mathcal{J}(u)\setminus N$. Hence we may assume without loss of generality that no such $a$, $R$, exists.
In this case, we fix a point $a\in Y_0$ and note by the co-area formula
Lemma~\ref{lem:coarea} applied to the real-valued function $d_a\circ u$ 
of bounded variation given by $x\mapsto d_Y(a,u(x))$, that
\[
\int_0^\infty P(u^{-1}(B(a,t)), X)\, dt=\Vert D\, d_a\circ u\Vert(X)\le \Vert D_{AM}u\Vert(X)<\infty.
\]
It follows that for each positive integer $n$ we can find $R_n>n$ such that 
$P(u^{-1}(B(a,R_n)), X)<1/n$. We now enlarge the null set $N$, chosen in 
the proof of Proposition~\ref{prop:finitejumps}, by replacing $N$ with 
\[
N\cup\bigcup_{k\in\N}\partial_*u^{-1}(B(a,R_k))\setminus\Sigma_\gamma u^{-1}(B(a,R_k)).
\]
We now fix $x\in\mathcal{J}(u)\setminus N$. Then, with 
$x\in \mathcal{J}(u)\setminus N$ as in the proof of Proposition~\ref{prop:finitejumps}, we have one of two cases:
\begin{enumerate}
\item[(a)] For each positive integer $n$ we have that
\[
\limsup_{r\to 0^+}\frac{\mu(B(x,r)\cap u^{-1}(\overline{B}(a,R_n)))}{\mu(B(x,r))}=0.
\]
\item[(b)] There is some positive integer $n_0$ such that for each $n\ge n_0$ we have 
\[
\limsup_{r\to 0^+}\frac{\mu(B(x,r)\cap u^{-1}(\overline{B}(a,R_n)))}{\mu(B(x,r))}>0.
\]
\end{enumerate}
Should Case~(a) happen, we say that $u$ is approximately continuous at $x$ with approximate limit $\infty$.
Such points form a $\mu$-measure null subset of $X$ because, by embedding $Y$ into a Banach space
and using Bochner integrals, we know that $\mu$-a.e.~point in $X$ is a Lebesgue point of $u$ as
$u\in L^1(X:V)$; note that the value of the function at a Lebesgue point must necessarily be a point in the Banach space
and hence cannot be infinite in nature.
We can include them in the set of approximately continuous
points of $u$. Thus it suffices to take care of Case~(b). In this case, we focus on covering the compact set
$\overline{B}(a,R_n)$ for some fixed $n\ge n_0$ by balls $B(y_i,\eps)$, $i=1,\cdots, N_\eps$, where implicitly
$N_\eps$ now depends on the choice of $R_n$ as well, but as $n$ is fixed, this dependence is suppressed. 
Here we ensure that $0<\eps<R_n/10$.
In so doing, we find one point, say $y_1$, such that 
\[
\limsup_{r\to 0^+}\frac{\mu(B(x,r)\cap u^{-1}(B(y_1,\eps)))}{\mu(B(x,r))}>0.
\]
Thus we can choose $E_1=u^{-1}(B(y_1,\eps))$, and as $x$ is not a point of
approximate continuity of $u$, we also know that
\[
\limsup_{r\to 0^+}\frac{\mu(B(x,r)\setminus u^{-1}(B(y_1,\eps)))}{\mu(B(x,r))}>0.
\]
If we also have
\[
\limsup_{r\to 0^+}\frac{\mu(B(x,r)\setminus u^{-1}(B(a,R_{2n})))}{\mu(B(x,r))}>0,
\]
then necessarily $x\in \partial_*u^{-1}B(a,R_{2n})$ and so as $x\not\in N$, we must have
that
\[
\liminf_{r\to 0^+}\frac{\mu(B(x,r)\setminus u^{-1}(B(a,R_{2n})))}{\mu(B(x,r))}\ge \gamma.
\]
If for all positive integers $n$ the above density property holds for $u^{-1}(B(a,R_{2n}))$, then
we can consider $\infty$ to be one of the jump values of $u$ at $x$. Continuing the argument
found in the proof of Proposition~\ref{prop:finitejumps} by covering $\overline{B}(a,\tfrac{11}{10}R_n)$
by balls of radius $\eps/6^2$ to find $y_2$, and proceeding from there to find a sequence $y_i\in Y$
that converges to $y_\infty\in Y$, we see that $y_\infty$ must also be a jump value of $u$ at $x$.
The rest of the argument as found in the proof of Proposition~\ref{prop:finitejumps} holds, as long as 
we consider $\infty$ to be one of the jump values if necessary.

If $\infty$ is a jump value of $u$ at $x$, then we must necessarily have that 
$x\in\Sigma_\gamma u^{-1}(B(a,R_k))$ for each $k$. As 
\[
1/k>P(u^{-1}(B(a,R_k)), X)\approx\mathcal{H}^{-1}(\Sigma_\gamma u^{-1}(B(a,R_k))),
\]
we must have that 
\[
\mathcal{H}^{-1}(\bigcap_k \Sigma_\gamma u^{-1}(B(a,R_k)))=0.
\]
That is, the collection of all points $x\in\mathcal{J}(u)\setminus N$ that have $\infty$ as a jump value must
be of $\mathcal{H}^{-1}$-measure zero as well. All other points in $\mathcal{J}(u)$ can be handled by the 
proof of Proposition~\ref{prop:finitejumps} by using covering arguments only for the compact set $\overline{B}(a,R_j)$ for
sufficiently large $j$.

\noindent {\bf Address:} \\

\noindent I.C.: Departamento de An\'alisis Matem\'atico y Matem\'atica Aplicada, Facultad de Ciencias
Matem\'aticas, Universidad Complutense de Madrid, 28040 Madrid, Spain.\\
\noindent E-mail: {\tt ivancaam@ucm.es}\\

\vskip .2cm

\noindent J.K.: Department of Mathematical Sciences, P.O.~Box 210025, University of Cincinnati, Cincinnati, OH~45221-0025, U.S.A.
\\
\noindent E-mail: {\tt klinejp@mail.uc.edu}\\

\vskip .2cm

\noindent N.S.: Department of Mathematical Sciences, P.O.~Box 210025, University of Cincinnati, Cincinnati, OH~45221-0025, U.S.A.\\
\noindent E-mail:  N.S.: {\tt shanmun@uc.edu}\\

\end{document}